\documentclass[conference]{IEEEtran}
\IEEEoverridecommandlockouts

\usepackage{amsmath, amssymb, amsfonts}
\usepackage{algorithm, algorithmic}
\usepackage{graphicx}
\usepackage{textcomp}
\usepackage{xcolor}
\usepackage{flushend} 
\usepackage{threeparttable}
\allowdisplaybreaks[4]

\newtheorem{theorem}{Theorem}
\newtheorem{lemma}{Lemma}
\newtheorem{corollary}{Corollary}
\newtheorem{assumption}{Assumption}

\newtheorem{remark}{Remark}
\usepackage{booktabs} 
\usepackage{multirow} 
\usepackage{makecell} 
\usepackage{cite}
\def\BibTeX{{\rm B\kern-.05em{\sc i\kern-.025em b}\kern-.08em
    T\kern-.1667em\lower.7ex\hbox{E}\kern-.125emX}}
\begin{document}

\title{Improved Convergence for Decentralized Stochastic Optimization with Biased Gradients\\
\thanks{This work was supported in part by the National Natural Science Foundation of China 
under Grant 62503344 and in part by the Postdoctoral Fellowship Program of CPSF under Grant GZC20241120. (Corresponding author: Yiwei Liao.)}
}

\author{
    \IEEEauthorblockN{Qing Xu, Yiwei Liao, Wenqi Fan, Xingxing You, and  Songyi Dian}
    \IEEEauthorblockA{School of Electrical Engineering, Sichuan University, Chengdu, China \\
    xuqing382015@163.com, liaoyiwei@scu.edu.cn, fanjessie159@gmail.com, youxingxing@scu.edu.cn, scudiansy@scu.edu.cn}
}

\maketitle

\begin{abstract}
Decentralized stochastic optimization has emerged as a fundamental paradigm for large-scale machine learning. However, practical implementations often rely on biased gradient estimators arising from communication compression or inexact local oracles, which severely degrade convergence in the presence of data heterogeneity. To address the challenge, we propose Decentralized Momentum Tracking with Biased Gradients (Biased-DMT), a novel decentralized algorithm designed to operate reliably under biased gradient information. We establish a comprehensive convergence theory for Biased-DMT in nonconvex settings and show that it achieves linear speedup with respect to the number of agents. The theoretical analysis shows that Biased-DMT decouples the effects of network topology from data heterogeneity, enabling robust performance even in sparse communication networks. Notably, when the gradient oracle introduces only absolute bias, the proposed method eliminates the structural heterogeneity error and converges to the exact physical error floor. For the case of relative bias, we further characterize the convergence limit and show that the remaining error is an unavoidable physical consequence of locally injected noise. Extensive numerical experiments corroborate our theoretical analysis and demonstrate the practical effectiveness of Biased-DMT across a range of decentralized learning scenarios.

\end{abstract}


\section{Introduction}
Decentralized stochastic optimization has emerged as a fundamental paradigm for solving large-scale machine learning and control problems over multi-agent networks. We consider a system of $n$ agents interacting over a connected network, aiming to collaboratively solve the following  optimization problem
\begin{equation}\label{problem}
    \min_{\mathbf{x} \in \mathbb{R}^d} F(\mathbf{x}) := \frac{1}{n} \sum_{i=1}^n f_i(\mathbf{x}),
\end{equation}
where $f_i: \mathbb{R}^d \to \mathbb{R}$ is a smooth, potentially nonconvex local objective function accessible only to agent $i$. Unlike centralized approaches (e.g., Parameter Server) where a central node aggregates information from all workers, decentralized algorithms allow agents to communicate only with their immediate neighbors via a peer-to-peer protocol \cite{c1}, \cite{c2}. This architecture effectively avoids the communication bottleneck at the central server and enhances data privacy, making it particularly suitable for applications ranging from distributed sensing to federated learning \cite{c3}.

The cornerstone of decentralized stochastic optimization is Decentralized Stochastic Gradient Descent (DSGD) \cite{c3}, \cite{c4}. While DSGD mimics the behavior of centralized SGD, its performance degrades significantly when the local data distributions are non-IID (heterogeneous). In such cases, the variance among local gradients introduces a nonvanishing steady-state error. To mitigate the issue, Gradient Tracking (GT) methods \cite{c5}, \cite{c6} were introduced to estimate the global average gradient, thereby ensuring exact convergence even with heterogeneous data. Furthermore, to accelerate convergence, momentum-based techniques have been integrated into decentralized schemes. However, standard decentralized momentum SGD (DSGDm) \cite{c8} still suffers from data heterogeneity. To address the challenge, recent advances have introduced momentum tracking algorithms \cite{c9}, \cite{c12}. By coordinating momentum updates across agents, Momentum Tracking combines the acceleration benefits of momentum with the robustness of tracking mechanisms, achieving optimal convergence rates that are independent of data heterogeneity.

Despite these significant advances, the aforementioned algorithms typically operate under a strong assumption: agents have access to \textit{unbiased} stochastic gradients of their local functions. In many practical scenarios, however, this assumption is violated. Agents often interact with \textit{biased} gradient oracles due to system constraints or privacy requirements. {For instance, in communication-constrained networks, agents may apply biased compression operators to gradients to reduce bandwidth usage \cite{c13, c14, c15, c16}. Similarly, in zeroth-order optimization, gradient estimates inherently carry systematic bias \cite{c17}.} A recent study by Jiang et al. \cite{c18} analyzed DSGD with biased gradients, revealing that the systematic bias introduces an irreducible error floor.  While Jiang et al. \cite{c18} rigorously analyzed the impact of bias, the  convergence rate is still heavily impacted by the data heterogeneity term. Conversely, existing Momentum Tracking methods \cite{c9} are robust to heterogeneity but lack theoretical guaranties when the gradient inputs are systematically biased.

In this paper, we bridge the critical gap by proposing decentralized momentum tracking with biased gradients, termed Biased-DMT. Then, we establish a rigorous theoretical foundation for its convergence and empirically validate its robust performance against both data heterogeneity and systematic bias. The main contributions of this work are summarized as follows.
\begin{itemize}
    \item \textbf{Algorithm Design and Unified Bound.} We propose Biased-DMT, a novel algorithm that structurally integrates biased gradient estimators into a momentum tracking framework. We establish a rigorous nonconvex convergence bound that elegantly accommodates both relative bias ($M_f$) and absolute bias ($\sigma_f^2$), upgrading the theoretical limits of optimization under imperfect oracles.
    
    \item \textbf{Linear Speedup and Topology Decoupling.} With an appropriately tuned momentum parameter, Biased-DMT absorbs the transient network-induced penalty and achieves the optimal $\mathcal{O}(1/\sqrt{nT})$ linear speedup. The resulting convergence bound explicitly decouples the network spectral gap $\rho$ from the data heterogeneity variance $\zeta^2$. Under absolute bias ($M_f = 0$), the structural heterogeneity error is  eliminated, enabling convergence to the exact physical error floor without requiring the commonly assumed bounded heterogeneity condition.
    
    \item \textbf{Empirical Validation of Theoretical Findings.} Extensive numerical experiments systematically validate our theoretical framework. The results confirm the algorithm's superiority in highly heterogeneous environments and explicitly verify the theoretically predicted stair-step error floors and steady-state dynamics under biased oracles.
\end{itemize}

\section{Related Work}

\subsection{Decentralized Optimization and Momentum Tracking}
Decentralized optimization has been extensively studied in recent years. The most fundamental algorithm, Decentralized SGD (D-PSGD) \cite{c3}, enables agents to optimize a global objective by averaging models with neighbors. However, D-PSGD suffers from a nonvanishing steady-state error caused by \textit{data heterogeneity} (the variance of local gradients, denoted as $\zeta^2$). It means that  D-PSGD cannot converge to the exact stationary point for constant step sizes, if the data is non-IID.

To address the limitation, Gradient Correction and Gradient Tracking methods, such as $D^2$ \cite{c4}, GT-DSGD [6] and GNSD [7], were proposed. These algorithms introduce correction mechanisms or auxiliary variables to track the global average gradient, successfully eliminating the heterogeneity-induced bias and ensuring exact convergence. However, standard GT and correction methods typically do not incorporate momentum acceleration, which limits their practical convergence speed, especially in nonconvex deep learning tasks.

Integrating momentum into decentralized schemes is a standard practice for acceleration. Decentralized Momentum SGD (DSGDm) \cite{c8} achieves linear speedup, but   it is not robust to data heterogeneity like D-PSGD. It motivated the development of \textit{Momentum Tracking} algorithms (MT), such as STEM \cite{c10} and the work of Takezawa et al. \cite{c9}. These methods align the momentum updates across agents, achieving both acceleration and robustness (i.e., removing $\zeta^2$ from the error floor). Nevertheless, these works strictly assume access to \textit{unbiased} stochastic gradients, which restricts their applicability in  communication, privacy protection and other learning-based scenarios.

\subsection{Optimization with Biased Gradients}
In many realistic distributed systems, agents usually suffer from \textit{biased} gradient estimators. Common sources of bias include gradient compression used to reduce communication overhead \cite{c13, c14, c15, c16}, or zeroth-order gradient estimation \cite{c17}. Theoretical analyses of SGD with biased gradients \cite{c19} showed that the systematic bias (denoted as $\sigma_f^2$) appears explicitly in the convergence bound.

Most relevant to our work is the recent study by Jiang et al. \cite{c18}, which provided a comprehensive analysis of \textit{Biased-DSGD}. It showed that the algorithm converges to an error floor determined by the bias variance. However, its framework is built upon the standard D-PSGD architecture and the asymptotic error bound still contains the data heterogeneity term $\zeta^2$. It implies that in highly heterogeneous environments, the performance of Biased-DSGD may degrade significantly.

\subsection{Summary of Theoretical Comparisons}
Table \ref{tab:comparison} highlights the theoretical advantages of Biased-DMT against representative baselines:

\textbf{1) Exact Linear Speedup and Noise Absorption.} Biased-DMT completely absorbs the network topology penalty and the pure stochastic noise variance ($\sigma^2$), achieving the exact $\mathcal{O}(1/\sqrt{nT})$ linear speedup. 

\textbf{2) Topology-Heterogeneity Decoupling.} Existing methods \cite{c3, c8, c18} severely amplify data heterogeneity $\zeta^2$ by the spectral gap $\rho$. In contrast, Biased-DMT strictly decouples $\rho$ from $\zeta^2$, leaving an irreducible physical error floor of only $\mathcal{O}(M_f \zeta^2 + \sigma_f^2)$.

\textbf{3) Heterogeneity Independence.} When the oracle exhibits only absolute bias ($M_f = 0$), the residual heterogeneity term perfectly vanishes. Biased-DMT attains the exact $\mathcal{O}(\sigma_f^2)$ error floor without requiring the commonly used bounded data heterogeneity assumption.

\begin{table}[htbp]
    \centering
    \caption{Comparison of convergence bounds.}
    \label{tab:comparison}
    \renewcommand{\arraystretch}{1.5} 
    \setlength{\tabcolsep}{3pt} 
    \resizebox{\columnwidth}{!}{
    \begin{tabular}{l c c l}
        \toprule[1.2pt]
        \textbf{Algorithm} & \textbf{Mom.} & \textbf{Biased} & \textbf{Convergence Bound} \\
        \midrule[0.6pt]
        D-PSGD \cite{c3} & $\times$ & $\times$ & $\mathcal{O}\left( \frac{1}{\sqrt{nT}} \right) + \mathcal{O}\left( \frac{\zeta^2}{\rho} \right)$ \\
        GT-DSGD \cite{c6} & $\times$ & $\times$ & $\mathcal{O}\left( \frac{1}{\sqrt{nT}} \right)$ \\
        DSGDm \cite{c8} & \checkmark & $\times$ & $\mathcal{O}\left( \frac{1}{\sqrt{nT}} \right) + \mathcal{O}\left( \frac{\zeta^2}{\rho} \right)$ \\
        Biased-DSGD \cite{c18} & $\times$ & \checkmark & $\mathcal{O}\left( \frac{1}{\sqrt{nT}} + \frac{n}{T} \right) + \mathcal{O}\left( \frac{\zeta^2}{\rho^2} +M_f \zeta^2 +\sigma_f^2 \right)$ \\
        \midrule[0.6pt]
        \textbf{Biased-DMT} (Relative Bias) & \textbf{\checkmark} & \textbf{\checkmark} & $\boldsymbol{\mathcal{O}\left( \frac{1}{\sqrt{nT}} \right) + \mathcal{O}\left( M_f \zeta^2 + \sigma_f^2 \right)}$ \\
        \bottomrule[1.2pt]
    \end{tabular}
    }
    \vspace{0.2cm} 
    \\
    \raggedright
    {\scriptsize 
    \begin{tablenotes}
    \item \textbf{Note:} Constants omitted. $\rho$: spectral gap; $\zeta^2$: data heterogeneity; $\sigma_f^2, M_f$: absolute and relative bias.
    \end{tablenotes}}
\end{table}

\section{Preliminaries}

In this section, we formally present the notations and detail the proposed Biased-DMT algorithm. Subsequently, we explicitly state the standard assumptions required for our theoretical analysis.

\subsection{Notations}
We use boldface lowercase letters (e.g., $\mathbf{x}$) for vectors and boldface uppercase letters (e.g., $\mathbf{X}$) for matrices. $\|\cdot\|_F$ denotes the Frobenius norm. $\mathbf{1}_n$ is the column vector of $n$ ones, and $\mathbf{I}_n$ is the identity matrix. $\mathbf{X}^t = [\mathbf{x}_1^t, \dots, \mathbf{x}_n^t] \in \mathbb{R}^{d \times n}$ denotes the model parameters of all $n$ agents at iteration $t$, and $\bar{\mathbf{x}}^t = \frac{1}{n} \mathbf{X}^t \mathbf{1}_n \in \mathbb{R}^d$ is their average parameter vector. To facilitate the matrix-form analysis, we define the corresponding average matrix as $\bar{\mathbf{X}}^t := \bar{\mathbf{x}}^t \mathbf{1}_n^\top \in \mathbb{R}^{d \times n}$. Equivalently, it can be compactly written as $\bar{\mathbf{X}}^t = \mathbf{X}^t\mathbf{J}$, where $\mathbf{J} = \frac{1}{n}\mathbf{1}_n\mathbf{1}_n^\top$ is the averaging matrix. The matrix notation naturally extends to other variables such as $\bar{\mathbf{V}}^t$ and $\bar{\mathbf{M}}^t$. $\nabla \mathbf{F}(\mathbf{X}^t) = [\nabla f_1(\mathbf{x}_1^t), \dots, \nabla f_n(\mathbf{x}_n^t)]$ represents the matrix of true local gradients. We use $\tilde{g}_i(\cdot)$ to denote the \textit{biased} stochastic gradient estimator accessible to agent $i$, which approximates the true gradient $\nabla f_i(\cdot)$.

\subsection{Proposed Algorithm: Biased-DMT}
We formally introduce \textbf{Biased-DMT} to solve problem (1), with the detailed procedure described in Algorithm 1. Each agent $i$ maintains three local variables: the model parameter $x_i$, the momentum estimator $m_i$, and the tracking variable $v_i$. 

In addition, the update logic differs from traditional SGD in order to guarantee theoretical consistency: agents query the biased stochastic gradient $\tilde{g}_i$ at the \textit{updated} intermediate model location $x_i^{t+1}$ instead of $x_i^t$.  The subsequent momentum and tracker updates then fuse the new gradient information with the network consensus direction.

\begin{algorithm}[htbp]
    \caption{Biased-DMT}
    \label{alg:biased_dmt}
    \begin{algorithmic}[1]
        \STATE \textbf{Input:} Initial point $\mathbf{x}_i^0 = \bar{\mathbf{x}}^0$, step size $\eta > 0$, momentum coefficient $\lambda \in (0,1]$.
        \STATE \textbf{Initialize:} $\mathbf{m}_i^0 = \tilde{g}_i(\mathbf{x}_i^0)$, $\mathbf{v}_i^0 = \mathbf{m}_i^0$ for all agents $i \in \{1, \dots, n\}$.
        \FOR{$t = 0, 1, \dots, T-1$}
            \FOR{each agent $i \in \{1, \dots, n\}$ in parallel}
                \STATE $\triangleright$ \textit{Step 1: Consensus \& Model Update}
                \STATE Receive $\mathbf{x}_j^t$ from neighbors.
                \STATE $\mathbf{x}_i^{t+1} = \sum_{j=1}^n w_{ij} \mathbf{x}_j^t - \eta \mathbf{v}_i^t$
                
                \STATE $\triangleright$ \textit{Step 2: Biased Gradient Query}
                \STATE Sample biased gradient $\tilde{g}_i(\mathbf{x}_i^{t+1})$.
                
                \STATE $\triangleright$ \textit{Step 3: Momentum \& Tracking Update}
                \STATE $\mathbf{m}_i^{t+1} = (1-\lambda)\mathbf{m}_i^t + \lambda \tilde{g}_i(\mathbf{x}_i^{t+1})$
                \STATE Receive $\mathbf{v}_j^t$ from neighbors.
                \STATE $\mathbf{v}_i^{t+1} = \sum_{j=1}^n w_{ij} \mathbf{v}_j^t + \mathbf{m}_i^{t+1} - \mathbf{m}_i^t$
            \ENDFOR
        \ENDFOR
    \end{algorithmic}
\end{algorithm}

\textbf{Matrix Form Representation.}
To facilitate the theoretical analysis, we rewrite the item-wise updates of Algorithm \ref{alg:biased_dmt} into a compact matrix form. Let $\tilde{\mathbf{G}}(\mathbf{X}^{t+1}) := [\tilde{g}_1(\mathbf{x}_1^{t+1}), \dots, \tilde{g}_n(\mathbf{x}_n^{t+1})]$ denote the matrix of biased stochastic gradients. The system dynamics evolve as follows
\begin{subequations} \label{eq:matrix_form}
\begin{align}
    \mathbf{X}^{t+1} &= \mathbf{X}^t \mathbf{W} - \eta \mathbf{V}^t, \label{eq:update_X} \\
    \mathbf{M}^{t+1} &= (1-\lambda)\mathbf{M}^t + \lambda \tilde{\mathbf{G}}(\mathbf{X}^{t+1}), \label{eq:update_M} \\
    \mathbf{V}^{t+1} &= \mathbf{V}^t \mathbf{W} + \mathbf{M}^{t+1} - \mathbf{M}^t. \label{eq:update_V}
\end{align}
\end{subequations}
Here, (\ref{eq:update_X}) represents the consensus step combined with the descent direction. Equation (\ref{eq:update_M}) is the local momentum update, and (\ref{eq:update_V}) ensures that the tracker $\mathbf{V}^t$ aligns with the average momentum direction.

\subsection{Assumptions}
The convergence analysis relies on the following standard assumptions regarding the objective function, network topology, consistent with \cite{c3}, \cite{c9}, \cite{c18}, and the biased oracle.

\begin{assumption}[Smoothness and Lower Bound] \label{ass:smoothness}
    Each local objective function $f_i(\mathbf{x})$ is $L$-smooth, meaning there exists a constant $L > 0$ such that for all $\mathbf{x}, \mathbf{y} \in \mathbb{R}^d$,
    \begin{equation}
        \|\nabla f_i(\mathbf{x}) - \nabla f_i(\mathbf{y})\| \le L \|\mathbf{x} - \mathbf{y}\|.
    \end{equation}
    Furthermore, the global objective function $F(\mathbf{x})$ is bounded from below, i.e., $F^* := \inf_{\mathbf{x}} F(\mathbf{x}) > -\infty$.
\end{assumption}

\begin{assumption}[Network Topology] \label{ass:network}
    The agents communicate over a connected graph endowed with a mixing matrix $\mathbf{W} = [w_{ij}] \in \mathbb{R}^{n \times n}$. The matrix $\mathbf{W}$ satisfies
    \begin{enumerate}
        \item \textbf{Doubly Stochastic.} The mixing matrix $\mathbf{W}$ is doubly stochastic, satisfying $\mathbf{W} \mathbf{1}_n = \mathbf{1}_n$ and $\mathbf{1}_n^\top \mathbf{W} = \mathbf{1}_n^\top$.
        \item \textbf{Spectral Gap.} The eigenvalues of $\mathbf{W}$ are real and sorted as $1 = \lambda_1 > |\lambda_2| \ge \dots \ge |\lambda_n|$. We define the spectral gap as $\rho := 1 - |\lambda_2| \in (0, 1]$.
    \end{enumerate}
\end{assumption}

\begin{remark}
    Assumption \ref{ass:network} is fundamental for consensus algorithms. It implies the contraction property $\|\mathbf{X}\mathbf{W} - \bar{\mathbf{x}}\mathbf{1}_n^\top\|_F \le (1-\rho)\|\mathbf{X} - \bar{\mathbf{x}}\mathbf{1}_n^\top\|_F$, which ensures that agents' local models converge to the global average.
\end{remark}

Unlike traditional settings that assume unbiased gradient estimators, we consider a more general scenario where the gradient oracle is biased.

\begin{assumption}[Biased Gradient Oracle] \label{ass:oracle}
    At each iteration $t$, each agent $i$ has access to a stochastic gradient oracle $\tilde{g}_i(\mathbf{x})$ satisfying
    \begin{enumerate}
        \item \textbf{Bounded Variance.} The variance of the stochastic gradient is bounded by $\sigma^2 \ge 0$,
        \begin{equation}
            \mathbb{E}[\|\tilde{g}_i(\mathbf{x}) - \mathbb{E}[\tilde{g}_i(\mathbf{x})]\|^2 | \mathbf{x}] \le \sigma^2.
        \end{equation}
        \item \textbf{Bounded Bias.} The systematic bias is bounded by a relative term $M_f \ge 0$ and a constant $\sigma_f^2 \ge 0$,
        \begin{equation} \label{eq:bias_assumption}
            \|\mathbb{E}[\tilde{g}_i(\mathbf{x}) | \mathbf{x}] - \nabla f_i(\mathbf{x})\|^2 \le M_f \|\nabla f_i(\mathbf{x})\|^2 + \sigma_f^2.
        \end{equation}
    \end{enumerate}
\end{assumption}

\begin{remark}
    Assumption \ref{ass:oracle}  generalizes the standard unbiased setting (where $M_f=0, \sigma_f=0$). The term $M_f \|\nabla f_i(\mathbf{x})\|^2$ captures the relative bias proportional to the gradient norm, while $\sigma_f^2$ captures the absolute bias. The decomposition aligns with our theoretical analysis in Section IV.
\end{remark}


\section{Convergence Analysis}

In this section, we provide the theoretical analysis of the proposed algorithm. The analysis proceeds in three steps: first, we bound the errors related to the consensus and momentum tracking mechanisms; second, we analyze the descent property of the objective function; and finally, we derive the global convergence rate.

\subsection{Auxiliary Lemmas}

We begin by bounding the expected change in the momentum matrix, which is crucial for analyzing the tracking error.

\begin{lemma}[Bound on Momentum Change] \label{lemma:mom_change}
Suppose Assumptions 1 and 3 hold. The expected change in the momentum matrix is bounded by
\begin{align}
    &\mathbb{E}[\|\mathbf{M}^{t+1} - \mathbf{M}^t\|_F^2]\nonumber \\
    \le& 3\lambda^2 \mathcal{G}^t 
    + 3\lambda^2 L^2 (1 + 2M_f) \mathbb{E}[\|\mathbf{X}^{t+1} - \mathbf{X}^t\|_F^2] \nonumber \\
    & + 3\lambda^2 \big( n(\sigma^2 + \sigma_f^2) + 2M_f \mathbb{E}[\|\nabla \mathbf{F}(\mathbf{X}^t)\|_F^2] \big).
\end{align}
\end{lemma}
\noindent \textit{Proof:} See Appendix A.

Based on the momentum change bound, we can characterize the contraction properties of the tracking error $\Xi_v^t := \mathbb{E}[\|\mathbf{V}^t - \bar{\mathbf{V}}^t\|_F^2]$ and the consensus error $\Xi_x^t := \mathbb{E}[\|\mathbf{X}^t - \bar{\mathbf{X}}^t\|_F^2]$.

\begin{lemma}[Contraction of Tracking Error] \label{lemma:tracking_error}
Under Assumption 2,  the tracking error satisfies
\begin{equation}
    \Xi_v^{t+1} \le \left(1 - \rho\right) \Xi_v^t + \frac{1}{\rho} \mathbb{E}[\|\mathbf{M}^{t+1} - \mathbf{M}^t\|_F^2].
\end{equation}
\end{lemma}
\noindent \textit{Proof:} See Appendix B.

\begin{lemma}[Contraction of Consensus Error] \label{lemma:consensus_error}
Under Assumption 2,  the consensus error satisfies
\begin{equation}
    \Xi_x^{t+1} \le \left(1 - \rho\right) \Xi_x^t + \frac{\eta^2}{\rho} \Xi_v^t.
\end{equation}
\end{lemma}
\noindent \textit{Proof:} See Appendix C.

Next, we analyze the estimation errors introduced by the biased gradient oracle. For convenience, denote  $\hat{G}^t := \mathbb{E}[\|\nabla F(\mathbf{X}^t)\mathbf{1} - \mathbf{M}^t\mathbf{1}\|^2]$ and $\mathcal{G}^t := \mathbb{E}[\|\mathbf{M}^t - \nabla \mathbf{F}(\mathbf{X}^t)\|_F^2]$. 

\begin{lemma}[Recursion of Average Momentum Error] \label{lemma:avg_mom_error}
Under Assumptions 1 and 3, $\hat{G}^t$ satisfies
\begin{align}
    \hat{G}^{t+1} &\le (1-\lambda)\hat{G}^t + \left( \frac{2n}{\lambda} + 4\lambda n M_f \right) L^2\mathbb{E}[\|\mathbf{X}^{t+1} - \mathbf{X}^t\|_F^2] \nonumber \\
    &\quad + 4\lambda n M_f \mathbb{E}[\|\nabla \mathbf{F}(\mathbf{X}^t)\|_F^2] + 2\lambda n^2 \sigma_f^2 + \lambda^2 n \sigma^2.
\end{align}
\end{lemma}
\noindent \textit{Proof:} See Appendix D.

\begin{lemma}[Recursion of Momentum Estimation Error] \label{lemma:mom_est_error}
 Under Assumptions 1 and 3, $\mathcal{G}^t$ satisfies
\begin{align}
    \mathcal{G}^{t+1} &\le (1-\lambda)\mathcal{G}^t + \left( \frac{2}{\lambda} + 4\lambda M_f \right) L^2\mathbb{E}[\|\mathbf{X}^{t+1} - \mathbf{X}^t\|_F^2] \nonumber \\
    &\quad + 4\lambda M_f \mathbb{E}[\|\nabla \mathbf{F}(\mathbf{X}^t)\|_F^2] + 2\lambda n \sigma_f^2 + \lambda^2 n \sigma^2.
\end{align}
\end{lemma}
\noindent \textit{Proof:} See Appendix E.

\begin{lemma}[Bound on Parameter Difference] \label{lemma:param_diff}
Under Assumption 2, the expected squared change in the model parameters between two consecutive iterations can be bounded by the consensus error $\Xi_x^t$, the tracking error $\Xi_v^t$, the average momentum error $\hat{G}^t$, and the true global gradient norm
\begin{align}
    \mathbb{E}[\|\mathbf{X}^{t+1} - \mathbf{X}^t\|_F^2] &\le (12 + 9\eta^2 L^2) \Xi_x^t + 3\eta^2 \Xi_v^t \nonumber \\
    &\quad + \frac{9\eta^2}{n} \hat{G}^t + 9n\eta^2 \mathbb{E}[\|\nabla F(\bar{\mathbf{x}}^t)\|^2]. \label{eq:param_diff}
\end{align}
\end{lemma}
\textit{Proof:} See Appendix F.

\subsection{Main Convergence Results}
Using the $L$-smoothness property and leveraging the auxiliary bounds established in the previous subsection, we can establish the descent property of the global objective function.

\begin{lemma}[Descent Lemma] \label{lemma:descent}
Suppose Assumption 1 holds. For any step size $\eta \le \frac{1}{L}$, the expected objective function value satisfies
\begin{align}
    \mathbb{E}[F(\bar{\mathbf{x}}^{t+1})] &\le \mathbb{E}[F(\bar{\mathbf{x}}^t)] - \frac{\eta}{2} \mathbb{E}[\|\nabla F(\bar{\mathbf{x}}^t)\|^2] \nonumber \\
    &\quad - \frac{\eta}{2}(1 - L\eta) \mathbb{E}[\|\bar{\mathbf{m}}^t\|^2] \nonumber \\
    &\quad + \frac{\eta}{n^2} \hat{G}^t + \frac{\eta L^2}{n} \Xi_x^t.
\end{align}
\end{lemma}
\noindent \textit{Proof:} See Appendix G

\begin{remark}
Lemma \ref{lemma:descent} reveals the core mechanism of our algorithm. The first negative term $-\frac{\eta}{2} \mathbb{E}[\|\nabla F(\bar{\mathbf{x}}^t)\|^2]$ drives the convergence. The second negative term involving $\|\bar{\mathbf{m}}^t\|^2$ acts as a crucial buffer to absorb the corresponding positive momentum errors accumulated in Lemma \ref{lemma:param_diff}, provided that $\eta$ is sufficiently small. Crucially, the error terms $\hat{G}^t$ and $\Xi_x^t$ are scaled by $\frac{1}{n^2}$ and $\frac{1}{n}$ respectively, which is the mathematical key to achieving the linear speedup with respect to the network size $n$.
\end{remark}

To formally state our results, we define the data heterogeneity bound commonly used in decentralized literature: $\frac{1}{n} \sum_{i=1}^n \|\nabla f_i(\mathbf{x}) - \nabla F(\mathbf{x})\|^2 \le \zeta^2$ for all $\mathbf{x}$. Let the sequence $\{\mathbf{X}^t\}$ be generated by the Biased-DMT algorithm.

\begin{theorem}[Convergence Bound] \label{thm:main}
Suppose Assumptions 1, 2, and 3 hold. Assume the relative bias ratio is bounded such that $M_f \le \frac{1}{256}$.  If the momentum parameter $\lambda$ and the step size $\eta$ satisfy the topology-aware conditions $\lambda \le \frac{\rho}{4\sqrt{n}}$ and $\eta \le \min \left\{ \frac{1}{L}, \frac{\rho \lambda}{8L}, \frac{\lambda}{16L (1+M_f)} \right\}$, then the averaged expected squared gradient norm over $T$ iterations is explicitly bounded by
\begin{align} \label{eq:thm_explicit}
    \frac{1}{T} \sum_{t=0}^{T-1} \mathbb{E}[\|\nabla F(\bar{\mathbf{x}}^t)\|^2] &\le \frac{4\Phi^0}{\eta T} + 64M_f\left(1 + \frac{9\lambda^2 n}{4\rho^2}\right)\zeta^2 \nonumber \\
    &\quad + \frac{8\lambda}{n} \left( 1 + \frac{3\lambda n^2}{2\rho^2} + \frac{3\lambda^2 n^2}{2\rho^2} \right) \sigma^2 \nonumber \\
    &\quad + 16\left(1 + \frac{9\lambda^2 n}{4\rho^2}\right)\sigma_f^2,
\end{align}
where $\Phi^0$ is the initial value of the constructed Lyapunov function.
\end{theorem}
\textit{Proof:} See Appendix H. \hfill 

\begin{remark}
The explicit bound in \eqref{eq:thm_explicit} deeply reveals the dynamics of Biased-DMT and demonstrates its fundamental superiority over standard Biased-DSGD \cite{c18}.
\begin{itemize}
    \item \textbf{Decoupling of Heterogeneity.} In Biased-DSGD, data heterogeneity $\zeta^2$ scales with $\mathcal{O}(\frac{\zeta^2}{\rho^2})$. In contrast, our residual heterogeneity is strictly scaled by $M_f$. Under absolute bias ($M_f=0$), the $\zeta^2$ term vanishes entirely.
    \item \textbf{Variance Reduction of Pure Noise.} Unlike standard approaches where pure stochastic noise $\sigma^2$ and systematic bias $\sigma_f^2$ are heavily coupled, Biased-DMT perfectly isolates $\sigma^2$. The pure noise multiplier strictly scales with $\lambda$, which enables it to be completely absorbed into the transient rate without leaving an additional steady-state error floor.
    \item \textbf{Transient Acceleration.} The topology penalty in \eqref{eq:thm_explicit} is effectively absorbed by tuning $\lambda$. This decouples the step size $\eta$ from network constraints, allowing a more aggressive learning rate to shorten the transient phase.
\end{itemize}
\end{remark}

\begin{corollary}\label{cor:rate}
Under the conditions of Theorem \ref{thm:main}, we consider the standard absolute bias setting where the gradient oracle has no relative error ($M_f = 0$). Suppose the total number of iterations $T$ is sufficiently large such that $T \ge 16n^2 / \rho^2$. By selecting the dynamic momentum parameter $\lambda = \sqrt{\frac{n}{T}}$ and the step size $\eta = \frac{1}{16L}\sqrt{\frac{n}{T}}$, Biased-DMT achieves the linear speedup with respect to the network size $n$:
\begin{equation} \label{eq:cor_bound}
    \frac{1}{T} \sum_{t=0}^{T-1} \mathbb{E}[\|\nabla F(\bar{\mathbf{x}}^t)\|^2] \le \mathcal{O}\left( \frac{1}{\sqrt{nT}} \right) + \mathcal{O}\left( \sigma_f^2 \right).
\end{equation}

\end{corollary}

\section{Numerical Experiments}


In this section, we evaluate the empirical performance of the proposed Biased-DMT algorithm. The experiments are designed to verify our theoretical findings, particularly focusing on the algorithm's robustness to data heterogeneity and its resilience against biased gradient estimators.

\subsection{Experimental Setup}
\textbf{Dataset and Problem Formulation.} Following standard decentralized optimization benchmarks \cite{c9}, \cite{c10}, we consider a nonconvex binary classification problem using the \texttt{a9a} dataset from the LIBSVM library. The objective is to minimize a logistic regression loss regularized by a nonconvex term. The local objective function at agent $i$ is formulated as
\begin{equation}
    f_i(\mathbf{x}) = \frac{1}{m_i} \sum_{j=1}^{m_i} \log\big(1 + \exp(-y_{i,j} \mathbf{a}_{i,j}^\top \mathbf{x})\big) + \alpha \sum_{k=1}^d \frac{x_k^2}{1 + x_k^2},
\end{equation}
where $\mathbf{a}_{i,j} \in \mathbb{R}^d$ is the feature vector, $y_{i,j} \in \{-1, 1\}$ is the corresponding label, $m_i$ is the number of local samples, and the penalty parameter is set to $\alpha = 0.01$ to ensure nonconvexity.

\textbf{Network and Data Heterogeneity.} We simulate a decentralized network of $n = 20$ agents communicating over a ring topology, which naturally provides a doubly stochastic mixing matrix $W$. To construct a strictly heterogeneous (Non-IID) scenario and maximize the variance $\zeta^2$, the training samples are sorted by labels and sequentially partitioned among the agents.

\textbf{Biased Oracle and Settings.} To simulate the biased gradient oracle, we inject a systematic error into the local stochastic gradient, formulated as $\tilde{g}_i(\mathbf{x}) = \nabla f_i(\mathbf{x}; \xi_i) + \mathbf{e}_i$. Specifically, the bias vector $\mathbf{e}_i$ is drawn from a Gaussian distribution with a non-zero mean, i.e., $\mathbf{e}_i \sim \mathcal{N}(\boldsymbol{\mu}, \sigma_e^2 \mathbf{I})$. This construction effectively captures the inherent systematic gradient bias ($\sigma_f^2 > 0$) specified in Assumption \ref{ass:oracle}. All algorithms utilize a mini-batch size of $b=64$. The optimal step sizes $\eta$ and momentum coefficients $\lambda$ are carefully fine-tuned via grid search to ensure a fair comparison.

\begin{figure}[htbp]
    \centering
    \includegraphics[width=0.95\linewidth]{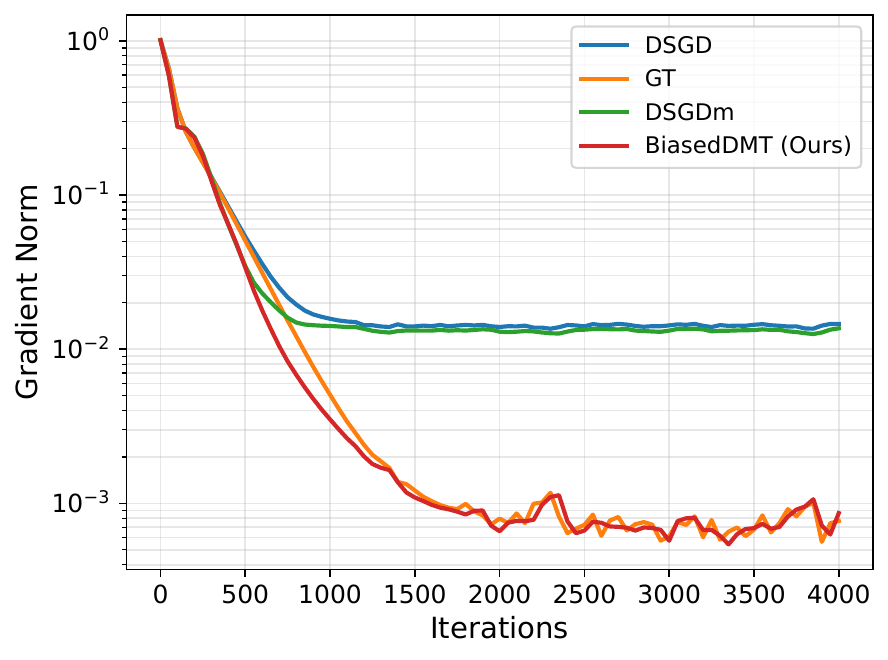}
    \caption{Convergence comparison under biased gradient oracle and Non-IID data.}
    \label{fig:comparison}
\end{figure}

\subsection{Performance under Biased Oracles}
Fig. \ref{fig:comparison} compares \textbf{Biased-DMT} against Biased-DSGD \cite{c18}, DSGDm \cite{c8}, and GT-DSGD \cite{c6} under identical biased settings.

DSGD and DSGDm fail to converge tightly, since they remain highly vulnerable to data heterogeneity ($\zeta^2$) in the Non-IID partition. Conversely, GT-DSGD mitigates heterogeneity but exhibits a slower, less smooth convergence trajectory without momentum acceleration. While Biased-DSGD theoretically handles gradient bias, its empirical steady-state error floor remains fundamentally bounded by $\zeta^2$.

In stark contrast, \textbf{Biased-DMT} consistently achieves the steepest descent rate and lowest training loss. This empirically corroborates our theoretical claim (Corollary 1): momentum tracking effectively nullifies the impact of data heterogeneity. Under the absolute bias setting, Biased-DMT completely eradicates the $\zeta^2$ error, leaving an optimal error floor dependent solely on the inherent bias $\sigma_f^2$.

\begin{figure}[htbp]
    \centering
    \includegraphics[width=0.95\linewidth]{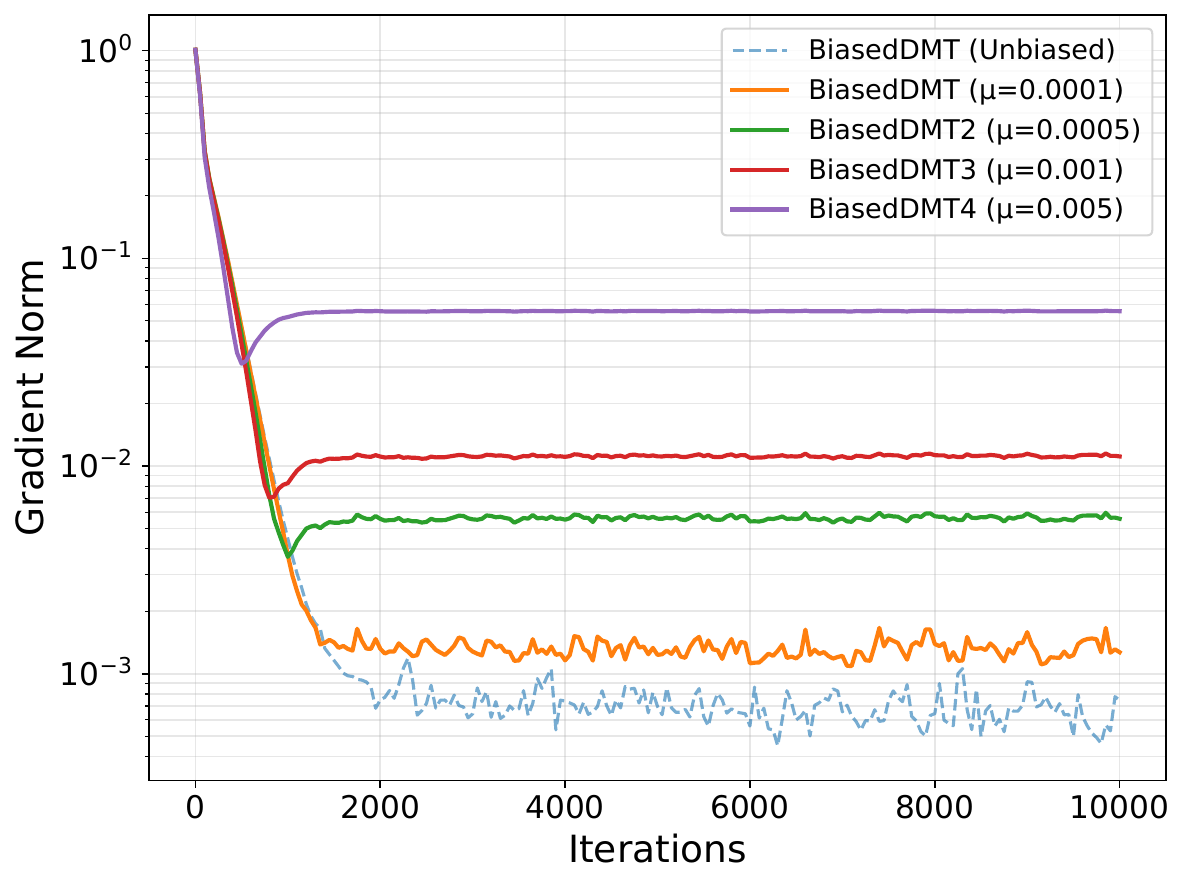}
    \caption{Convergence of Biased-DMT under varying gradient bias magnitudes.}
    \label{fig:biased_vs_unbiased}
\end{figure}

\subsection{Impact of Gradient Bias Magnitude}
Fig. \ref{fig:biased_vs_unbiased} isolates the impact of the biased oracle by comparing Biased-DMT under unbiased and increasingly biased settings.

During the initial transient phase, all variants exhibit remarkably similar rapid convergence. Biased-DMT maintains robust momentum acceleration; high bias even induces a temporary directional drift (overshooting) that accelerates escape from initial complex regions, demonstrating strong algorithmic resilience.

Approaching the steady state, the impact of bias emerges. Biased variants do not reach the unbiased baseline but stall at specific convergence floors, exhibiting a clear hierarchical degradation: larger bias magnitudes yield higher steady-state errors.This stair-step phenomenon perfectly aligns with Theorem 1 and Corollary 1. Under absolute bias, the asymptotic error bound strictly simplifies to $\mathcal{O}(\sigma_f^2)$, entirely free of data heterogeneity contamination. This strict correlation directly validates that our framework accurately captures the fundamental physical limits of optimization under biased oracles.

\section{Conclusion}
In this paper, we propose Biased-DMT, a decentralized momentum tracking algorithm specifically designed for decentralized stochastic optimization in the presence of biased gradient estimators. The theoretical analysis reveals that the proposed method effectively mitigates the adverse effects of this coupling. In particular, under absolute bias, the momentum tracking mechanism  eliminates the structural heterogeneity error, enabling convergence to the exact physical error floor without relying on bounded heterogeneity assumptions. In contrast, under relative bias, we rigorously characterize the convergence limit and show that the resulting residual error is intrinsic to decentralized learning systems with locally injected noise, rather than a consequence of algorithmic deficiency. Extensive numerical experiments corroborate the theoretical findings and demonstrate the effectiveness and robustness of Biased-DMT across heterogeneous and sparsely connected networks. Future work includes extending the framework to more general settings, such as directed communication graphs and time-varying network topologies.


\appendices
\section*{APPENDIX}

\subsection{Proof of Lemma \ref{lemma:mom_change}}

Recalling the update rule for the local momentum
\begin{equation}
    \mathbf{m}_i^{t+1} = (1-\lambda)\mathbf{m}_i^t + \lambda \tilde{g}_i(\mathbf{x}_i^{t+1})
\end{equation}
and subtracting $\mathbf{m}_i^t$ from both sides, we get
\begin{equation}
   \label{mp_m} \mathbf{m}_i^{t+1} - \mathbf{m}_i^t = \lambda (\tilde{g}_i(\mathbf{x}_i^{t+1}) - \mathbf{m}_i^t).
\end{equation}
Rewriting \eqref{mp_m}  in matrix form, we have $\mathbf{M}^{t+1} - \mathbf{M}^t = \lambda (\tilde{\mathbf{G}}(\mathbf{X}^{t+1}) - \mathbf{M}^t)$. Taking the squared Frobenius norm and expectation, we obtain
\begin{equation} \label{eq:mom_diff_start}
    \mathbb{E}[\|\mathbf{M}^{t+1} - \mathbf{M}^t\|_F^2] = \lambda^2 \mathbb{E}[\|\tilde{\mathbf{G}}(\mathbf{X}^{t+1}) - \mathbf{M}^t\|_F^2].
\end{equation}
To bound the term on the right-hand side, we decompose the error by introducing the true gradients at steps $t+1$ and $t$, i.e.,
\begin{align}
    \tilde{\mathbf{G}}(\mathbf{X}^{t+1}) - \mathbf{M}^t &= \underbrace{(\tilde{\mathbf{G}}(\mathbf{X}^{t+1}) - \nabla \mathbf{F}(\mathbf{X}^{t+1}))}_{T_1} \nonumber \\
    &\quad + \underbrace{(\nabla \mathbf{F}(\mathbf{X}^{t+1}) - \nabla \mathbf{F}(\mathbf{X}^t))}_{T_2} \nonumber \\
    &\quad + \underbrace{(\nabla \mathbf{F}(\mathbf{X}^t) - \mathbf{M}^t)}_{T_3}.
\end{align}
Using the inequality $\|A+B+C\|^2 \le 3\|A\|^2 + 3\|B\|^2 + 3\|C\|^2$, we obtain
\begin{align} \label{eq:mom_decomp}
    \mathbb{E}[\|\tilde{\mathbf{G}}(\mathbf{X}^{t+1}) - \mathbf{M}^t\|_F^2] &\le 3\mathbb{E}[\|T_1\|_F^2] + 3\mathbb{E}[\|T_2\|_F^2] \nonumber \\
    &\quad + 3\mathbb{E}[\|T_3\|_F^2].
\end{align}
Next, we bound each term individually.

\textbf{Bounding $T_1$ (Oracle Error).} Using the property $\mathbb{E}[\|\mathbf{X} - \mathbf{Y}\|^2] = \mathbb{E}[\|\mathbf{X} - \mathbb{E}[\mathbf{X}]\|^2] + \|\mathbb{E}[\mathbf{X}] - \mathbf{Y}\|^2$, we split the oracle error into variance and bias components. Applying Assumption 3, we get
\begin{align} \label{eq:bound_T1}
    \mathbb{E}[\|T_1\|_F^2] &\le \sum_{i=1}^n \mathbb{E} \left[ \sigma^2 + \|\mathbb{E}[\tilde{g}_i(\mathbf{x}_i^{t+1}) \mid \mathbf{x}_i^{t+1}] - \nabla f_i(\mathbf{x}_i^{t+1})\|^2 \right] \nonumber \\
    &\le \sum_{i=1}^n \mathbb{E} \left[ \sigma^2 + M_f \|\nabla f_i(\mathbf{x}_i^{t+1})\|^2 + \sigma_f^2 \right] \nonumber \\
    &= n(\sigma^2 + \sigma_f^2) + M_f \mathbb{E}[\|\nabla \mathbf{F}(\mathbf{X}^{t+1})\|_F^2].
\end{align}

\textbf{Bounding $T_2$ (Gradient Variation).} Using Assumption 1 ($L$-smoothness), we have
\begin{align} \label{eq:bound_T2}
    \mathbb{E}[\|T_2\|_F^2] &= \sum_{i=1}^n \mathbb{E}[\|\nabla f_i(\mathbf{x}_i^{t+1}) - \nabla f_i(\mathbf{x}_i^t)\|^2] \nonumber \\
    &\le L^2 \mathbb{E}[\|\mathbf{X}^{t+1} - \mathbf{X}^t\|_F^2].
\end{align}

\textbf{Bounding $T_3$ (Estimation Error).} By the definition of the momentum estimation error $\mathcal{G}^t$, we know
\begin{equation} \label{eq:bound_T3}
    \mathbb{E}[\|T_3\|_F^2] = \mathbb{E}[\|\nabla \mathbf{F}(\mathbf{X}^t) - \mathbf{M}^t\|_F^2] = \mathcal{G}^t.
\end{equation}

\textbf{Handling $\|\nabla \mathbf{F}(\mathbf{X}^{t+1})\|_F^2$.} To remove the dependency on $t+1$, we use the inequality $\|\mathbf{u}\|^2 \le 2\|\mathbf{v}\|^2 + 2\|\mathbf{u}-\mathbf{v}\|^2$ and $L$-smoothness, and take the expectation on both sides
\begin{align} \label{eq:grad_norm_handle}
    &\mathbb{E}[\|\nabla \mathbf{F}(\mathbf{X}^{t+1})\|_F^2]\nonumber \\
    \le &2\mathbb{E}[\|\nabla \mathbf{F}(\mathbf{X}^t)\|_F^2] 
    + 2\mathbb{E}[\|\nabla \mathbf{F}(\mathbf{X}^{t+1}) - \nabla \mathbf{F}(\mathbf{X}^t)\|_F^2] \nonumber \\
    \le &2\mathbb{E}[\|\nabla \mathbf{F}(\mathbf{X}^t)\|_F^2] + 2L^2 \mathbb{E}[\|\mathbf{X}^{t+1} - \mathbf{X}^t\|_F^2].
\end{align}

Finally, substituting (\ref{eq:bound_T1}), (\ref{eq:bound_T2}), (\ref{eq:bound_T3}), and (\ref{eq:grad_norm_handle}) back into (\ref{eq:mom_decomp}) and then into (\ref{eq:mom_diff_start}), we have
\begin{align}
    &\mathbb{E}[\|\mathbf{M}^{t+1} - \mathbf{M}^t\|_F^2] \nonumber \\
    \le& 3\lambda^2 L^2 \mathbb{E}[\|\mathbf{X}^{t+1} - \mathbf{X}^t\|_F^2] + 3\lambda^2 \mathcal{G}^t + 3\lambda^2 n(\sigma^2 + \sigma_f^2) \nonumber \\
    &+ 3\lambda^2 M_f \big( 2\mathbb{E}[\|\nabla \mathbf{F}(\mathbf{X}^t)\|_F^2] + 2L^2 \mathbb{E}[\|\mathbf{X}^{t+1} - \mathbf{X}^t\|_F^2] \big).
\end{align}
Rearranging the coefficients for $\mathbb{E}[\|\mathbf{X}^{t+1} - \mathbf{X}^t\|_F^2]$ completes the proof. \hfill $\square$

\subsection{Proof of Lemma \ref{lemma:tracking_error}}

Recall the update rule for the tracker $\mathbf{V}^{t+1} = \mathbf{V}^t \mathbf{W} + \mathbf{M}^{t+1} - \mathbf{M}^t$.
Multiplying both sides by the averaging matrix $\mathbf{J} = \frac{1}{n}\mathbf{1}_n\mathbf{1}_n^\top$, and noting that $\mathbf{W}\mathbf{J} = \mathbf{J}\mathbf{W} = \mathbf{J}$, we get the evolution of the average
\begin{align}
    \bar{\mathbf{V}}^{t+1} &= \mathbf{V}^{t+1}\mathbf{J} = \mathbf{V}^t \mathbf{W} \mathbf{J} + (\mathbf{M}^{t+1} - \mathbf{M}^t)\mathbf{J} \nonumber \\
    &= \bar{\mathbf{V}}^t + \bar{\mathbf{M}}^{t+1} - \bar{\mathbf{M}}^t.
\end{align}
Subtracting the average update from the tracker update, we obtain
\begin{align}
    \mathbf{V}^{t+1} - \bar{\mathbf{V}}^{t+1} &= (\mathbf{V}^t \mathbf{W} - \bar{\mathbf{V}}^t) \nonumber \\
    &\quad + (\mathbf{M}^{t+1} - \mathbf{M}^t)(\mathbf{I} - \mathbf{J}).
\end{align}
We now analyze the first term $(\mathbf{V}^t \mathbf{W} - \bar{\mathbf{V}}^t)$. Using the fact that $\bar{\mathbf{V}}^t = \mathbf{V}^t \mathbf{J}$ and the decomposition $\mathbf{W} = (\mathbf{W} - \mathbf{J}) + \mathbf{J}$, we have
\begin{align}
    \mathbf{V}^t \mathbf{W} - \bar{\mathbf{V}}^t &= \mathbf{V}^t \mathbf{W} - \mathbf{V}^t \mathbf{J} = \mathbf{V}^t (\mathbf{W} - \mathbf{J}) \nonumber \\
    &= (\mathbf{V}^t - \bar{\mathbf{V}}^t + \bar{\mathbf{V}}^t)(\mathbf{W} - \mathbf{J}) \nonumber \\
    &= (\mathbf{V}^t - \bar{\mathbf{V}}^t)(\mathbf{W} - \mathbf{J}) + \bar{\mathbf{V}}^t(\mathbf{W} - \mathbf{J}).
\end{align}
Since $(\mathbf{W} - \mathbf{J})$ has zero column sums (i.e., $\mathbf{1}_n^\top(\mathbf{W} - \mathbf{J}) = \mathbf{0}$), the term $\bar{\mathbf{V}}^t(\mathbf{W} - \mathbf{J}) = \mathbf{0}$. Thus, we get
\begin{align} \label{eq:spectral_step}
    \|\mathbf{V}^t \mathbf{W} - \bar{\mathbf{V}}^t\|_F &= \|(\mathbf{V}^t - \bar{\mathbf{V}}^t)(\mathbf{W} - \mathbf{J})\|_F \nonumber \\
    &\le \|\mathbf{W} - \mathbf{J}\|_2 \|\mathbf{V}^t - \bar{\mathbf{V}}^t\|_F.
\end{align}
By Assumption 2 (Spectral Gap), $\|\mathbf{W} - \mathbf{J}\|_2 = \rho(\mathbf{W} - \mathbf{J}) = 1 - \rho$. It implies
\begin{equation}
    \|\mathbf{V}^t \mathbf{W} - \bar{\mathbf{V}}^t\|_F^2 \le (1-\rho)^2 \|\mathbf{V}^t - \bar{\mathbf{V}}^t\|_F^2.
\end{equation}

Now, applying Young's Inequality $\|X+Y\|_F^2 \le (1+\beta)\|X\|_F^2 + (1+\frac{1}{\beta})\|Y\|_F^2$ with $\beta = \frac{\rho}{1-\rho}$ (assuming $\rho \in (0, 1)$; the case $\rho=1$ holds trivially as $\mathbf{W}=\mathbf{J}$ and errors drop to zero), we obtain the coefficients $1+\beta = \frac{1}{1-\rho}$ and $1+\frac{1}{\beta} = \frac{1}{\rho}$. Then, we have
\begin{align}
    &\|\mathbf{V}^{t+1} - \bar{\mathbf{V}}^{t+1}\|_F^2 \nonumber \\
    \le& \frac{1}{1-\rho} \|\mathbf{V}^t \mathbf{W} - \bar{\mathbf{V}}^t\|_F^2 
    + \frac{1}{\rho} \|(\mathbf{I} - \mathbf{J})(\mathbf{M}^{t+1} - \mathbf{M}^t)\|_F^2 \nonumber \\
    \le& \frac{1}{1-\rho}(1-\rho)^2 \|\mathbf{V}^t - \bar{\mathbf{V}}^t\|_F^2  + \frac{1}{\rho} \|\mathbf{M}^{t+1} - \mathbf{M}^t\|_F^2 \nonumber \\
    =& (1-\rho) \|\mathbf{V}^t - \bar{\mathbf{V}}^t\|_F^2 + \frac{1}{\rho} \|\mathbf{M}^{t+1} - \mathbf{M}^t\|_F^2.
\end{align}
Here we used the contraction property $\|\mathbf{I}-\mathbf{J}\|_2 \le 1$. Taking expectations yields the result
\begin{equation}
    \Xi_v^{t+1} \le (1 - \rho) \Xi_v^t + \frac{1}{\rho} \mathbb{E}[\|\mathbf{M}^{t+1} - \mathbf{M}^t\|_F^2].
\end{equation}
\hfill $\square$

\subsection{Proof of Lemma \ref{lemma:consensus_error}}

Recall the update rule for the model parameters $\mathbf{X}^{t+1} = \mathbf{X}^t \mathbf{W} - \eta \mathbf{V}^t$.
Multiplying both sides by the averaging matrix $\mathbf{J} = \frac{1}{n}\mathbf{1}_n\mathbf{1}_n^\top$, and using the property $\mathbf{W}\mathbf{J} = \mathbf{J}$, we obtain the update rule for the average variable, i.e.,
\begin{align}
    \bar{\mathbf{X}}^{t+1} &= \mathbf{X}^{t+1}\mathbf{J} = \mathbf{X}^t \mathbf{W} \mathbf{J} - \eta \mathbf{V}^t \mathbf{J} \nonumber \\
    &= \bar{\mathbf{X}}^t - \eta \bar{\mathbf{V}}^t.
\end{align}
Subtracting the average update from the parameter update yields the deviation
\begin{align}
    \mathbf{X}^{t+1} - \bar{\mathbf{X}}^{t+1} &= (\mathbf{X}^t \mathbf{W} - \eta \mathbf{V}^t) - (\bar{\mathbf{X}}^t - \eta \bar{\mathbf{V}}^t) \nonumber \\
    &= (\mathbf{X}^t \mathbf{W} - \bar{\mathbf{X}}^t) - \eta (\mathbf{V}^t - \bar{\mathbf{V}}^t).
\end{align}
Similar to the analysis in Lemma~\ref{lemma:tracking_error}, we decompose the first term using $\bar{\mathbf{X}}^t = \mathbf{X}^t \mathbf{J}$ and $\mathbf{W} = (\mathbf{W} - \mathbf{J}) + \mathbf{J}$, i.e.,
\begin{align}
    \mathbf{X}^t \mathbf{W} - \bar{\mathbf{X}}^t &= \mathbf{X}^t (\mathbf{W} - \mathbf{J} + \mathbf{J}) - \mathbf{X}^t \mathbf{J} \nonumber \\
    &= \mathbf{X}^t (\mathbf{W} - \mathbf{J}) + \bar{\mathbf{X}}^t - \bar{\mathbf{X}}^t \nonumber \\
    &= (\mathbf{X}^t - \bar{\mathbf{X}}^t + \bar{\mathbf{X}}^t)(\mathbf{W} - \mathbf{J}).
\end{align}
Since $\bar{\mathbf{X}}^t$ has identical columns, $\bar{\mathbf{X}}^t(\mathbf{W} - \mathbf{J}) = \mathbf{0}$. Thus, the first term is simplified by
\begin{equation}
    \mathbf{X}^t \mathbf{W} - \bar{\mathbf{X}}^t = (\mathbf{X}^t - \bar{\mathbf{X}}^t)(\mathbf{W} - \mathbf{J}).
\end{equation}
Taking the squared Frobenius norm and applying Assumption 2 (Spectral Gap, $\|\mathbf{W} - \mathbf{J}\|_2 = 1 - \rho$), we have
\begin{align}
    \|\mathbf{X}^t \mathbf{W} - \bar{\mathbf{X}}^t\|_F^2 &\le \|\mathbf{W} - \mathbf{J}\|_2^2 \|\mathbf{X}^t - \bar{\mathbf{X}}^t\|_F^2 \nonumber \\
    &= (1 - \rho)^2 \|\mathbf{X}^t - \bar{\mathbf{X}}^t\|_F^2.
\end{align}
Now, using Young's Inequality $\|A+B\|_F^2 \le (1+\beta)\|A\|_F^2 + (1+\frac{1}{\beta})\|B\|_F^2$ to the update equation and setting $\beta = \frac{\rho}{1-\rho}$, we get
\begin{align}
    &\|\mathbf{X}^{t+1} - \bar{\mathbf{X}}^{t+1}\|_F^2 \nonumber \\
    \le& \frac{1}{1-\rho} \|\mathbf{X}^t \mathbf{W} - \bar{\mathbf{X}}^t\|_F^2  + \frac{1}{\rho} \|-\eta (\mathbf{V}^t - \bar{\mathbf{V}}^t)\|_F^2 \nonumber \\
    \le& \frac{1}{1-\rho}(1-\rho)^2 \|\mathbf{X}^t - \bar{\mathbf{X}}^t\|_F^2  + \frac{\eta^2}{\rho} \|\mathbf{V}^t - \bar{\mathbf{V}}^t\|_F^2 \nonumber \\
    =& (1-\rho) \|\mathbf{X}^t - \bar{\mathbf{X}}^t\|_F^2 + \frac{\eta^2}{\rho} \|\mathbf{V}^t - \bar{\mathbf{V}}^t\|_F^2.
\end{align}
Taking expectations on both sides, we obtain the final recursion
\begin{equation}
    \Xi_x^{t+1} \le (1 - \rho) \Xi_x^t + \frac{\eta^2}{\rho} \Xi_v^t.
\end{equation}
\hfill $\square$

\subsection{Proof of Lemma \ref{lemma:avg_mom_error}}

We define the error vector $\Delta^{t+1}_1 = \nabla F(\mathbf{X}^{t+1})\mathbf{1} - \mathbf{M}^{t+1}\mathbf{1}$. 
To tightly bound this term and properly capture the variance reduction effect of the momentum tracking, we must mathematically decouple the zero-mean variance noise and the systematic bias \emph{before} applying any inequality relaxations. 

Conditioned on the current model parameters $\mathbf{X}^{t+1}$, we define the zero-mean random noise vector $\boldsymbol{\xi}_i$ and the deterministic bias vector $\mathbf{b}_i$ for each agent $i$ as
\begin{align*}
    \boldsymbol{\xi}_i &= \tilde{g}_i(\mathbf{x}_i^{t+1}) - \mathbb{E}[\tilde{g}_i(\mathbf{x}_i^{t+1})], \\
    \mathbf{b}_i &= \mathbb{E}[\tilde{g}_i(\mathbf{x}_i^{t+1})] - \nabla f_i(\mathbf{x}_i^{t+1}).
\end{align*}
By Assumption 3, we have $\mathbb{E}[\boldsymbol{\xi}_i]=\mathbf{0}$, $\mathbb{E}[\|\boldsymbol{\xi}_i\|^2] \le \sigma^2$, and $\|\mathbf{b}_i\|^2 \le M_f \|\nabla f_i(\mathbf{x}_i^{t+1})\|^2 + \sigma_f^2$. 
Let $\boldsymbol{\xi} = \sum_{i=1}^n \boldsymbol{\xi}_i$ and $\mathbf{b} = \sum_{i=1}^n \mathbf{b}_i$, the biased gradient oracle can be written as $\tilde{\mathbf{G}}(\mathbf{X}^{t+1})\mathbf{1} = \nabla F(\mathbf{X}^{t+1})\mathbf{1} + \mathbf{b} + \boldsymbol{\xi}$.

Substituting this relation and the momentum update rule $\mathbf{M}^{t+1} = (1-\lambda)\mathbf{M}^t + \lambda \tilde{\mathbf{G}}(\mathbf{X}^{t+1})$ into $\Delta^{t+1}_1$, we can decompose the error into four components
\begin{align} \label{eq:delta_decomp}
    \Delta^{t+1}_1 &= \nabla F(\mathbf{X}^{t+1})\mathbf{1} \nonumber- \Big( (1-\lambda)\mathbf{M}^t\mathbf{1} \nonumber \\
    &\qquad + \lambda (\nabla F(\mathbf{X}^{t+1})\mathbf{1} + \mathbf{b} + \boldsymbol{\xi}) \Big) \nonumber \\
    &= (1-\lambda) \underbrace{(\nabla F(\mathbf{X}^t)\mathbf{1} - \mathbf{M}^t\mathbf{1})}_{A} \nonumber \\
    &\quad + (1-\lambda) \underbrace{(\nabla F(\mathbf{X}^{t+1})\mathbf{1} - \nabla F(\mathbf{X}^t)\mathbf{1})}_{B} \nonumber \\
    &\quad - \lambda \mathbf{b} - \lambda \boldsymbol{\xi}.
\end{align}

Taking the squared Euclidean norm and expectation, since $\boldsymbol{\xi}$ is a zero-mean random noise conditioned on $\mathbf{X}^{t+1}$, its cross-terms with the deterministic components ($A$, $B$, and $\mathbf{b}$) strictly vanish (i.e., $\mathbb{E}[\langle (1-\lambda)A + (1-\lambda)B - \lambda \mathbf{b}, \boldsymbol{\xi} \rangle] = 0$). Therefore, the stochastic noise is perfectly decoupled, yielding
\begin{align} \label{eq:G_decomp}
    \mathbb{E}[\|\Delta^{t+1}_1\|^2] &= \mathbb{E}[\| (1-\lambda)A + (1-\lambda)B - \lambda \mathbf{b} \|^2] \nonumber \\
    &\quad + \lambda^2 \mathbb{E}[\|\boldsymbol{\xi}\|^2].
\end{align}

Next, we bound the deterministic part using Young's Inequality $\|X+Y\|^2 \le (1+\alpha)\|X\|^2 + (1+\frac{1}{\alpha})\|Y\|^2$ with $\alpha = \frac{\lambda}{1-\lambda}$
\begin{align} \label{eq:deterministic_part}
    &\mathbb{E}[\| (1-\lambda)A + ((1-\lambda)B - \lambda \mathbf{b}) \|^2] \nonumber \\
    &\le \frac{1}{1-\lambda} (1-\lambda)^2 \mathbb{E}[\|A\|^2] + \frac{1}{\lambda} \mathbb{E}[\|(1-\lambda)B - \lambda \mathbf{b}\|^2] \nonumber \\
    &\le (1-\lambda)\hat{G}^t + \frac{2}{\lambda} (1-\lambda)^2 \mathbb{E}[\|B\|^2] + \frac{2}{\lambda} \lambda^2 \mathbb{E}[\|\mathbf{b}\|^2] \nonumber \\
    &\le (1-\lambda)\hat{G}^t + \frac{2}{\lambda} \mathbb{E}[\|B\|^2] + 2\lambda \mathbb{E}[\|\mathbf{b}\|^2],
\end{align}
where we use the inequality $\|X-Y\|^2 \le 2\|X\|^2 + 2\|Y\|^2$ and $(1-\lambda)^2 \le 1$ in the last two steps. Substituting \eqref{eq:deterministic_part} back into \eqref{eq:G_decomp}, we obtain the consolidated intermediate bound
\begin{align} \label{eq:G_decomp_combined}
    \mathbb{E}[\|\Delta^{t+1}_1\|^2] &\le (1-\lambda)\hat{G}^t + \frac{2}{\lambda} \mathbb{E}[\|B\|^2] \nonumber \\
    &\quad + 2\lambda \mathbb{E}[\|\mathbf{b}\|^2] + \lambda^2 \mathbb{E}[\|\boldsymbol{\xi}\|^2].
\end{align}

Now, we bound terms $B$, $\mathbf{b}$, and $\boldsymbol{\xi}$ individually.

\textbf{Bounding the Drift Term ($B$).}
Using Jensen's inequality $\|\sum_{i=1}^n \mathbf{z}_i\|^2 \le n \sum_{i=1}^n \|\mathbf{z}_i\|^2$ and the $L$-smoothness Assumption $\|\nabla f_i(\mathbf{x}) - \nabla f_i(\mathbf{y})\| \le L\|\mathbf{x}-\mathbf{y}\|$, we have
\begin{align} \label{eq:bound_B}
    \mathbb{E}[\|B\|^2] &= \mathbb{E}\left[\left\| \sum_{i=1}^n (\nabla f_i(\mathbf{x}_i^{t+1}) - \nabla f_i(\mathbf{x}_i^{t})) \right\|^2\right] \nonumber \\
    &\le n \mathbb{E}\left[\sum_{i=1}^n \|\nabla f_i(\mathbf{x}_i^{t+1}) - \nabla f_i(\mathbf{x}_i^t)\|^2\right] \nonumber \\
    &\le n L^2 \mathbb{E}\left[\sum_{i=1}^n \|\mathbf{x}_i^{t+1} - \mathbf{x}_i^t\|^2\right] \nonumber \\
    &= n L^2 \mathbb{E}[\|\mathbf{X}^{t+1} - \mathbf{X}^t\|_F^2].
\end{align}

\textbf{Bounding the Bias Term ($\mathbf{b}$).}
Using Jensen's inequality and Assumption 3, the squared norm of the aggregate bias is bounded by
\begin{align} \label{eq:bound_bias}
    \mathbb{E}[\|\mathbf{b}\|^2] &= \mathbb{E}\left[ \left\| \sum_{i=1}^n \mathbf{b}_i \right\|^2 \right] \le n \sum_{i=1}^n \mathbb{E}[\|\mathbf{b}_i\|^2] \nonumber \\
    &\le n \sum_{i=1}^n \mathbb{E}\left[ M_f \|\nabla f_i(\mathbf{x}_i^{t+1})\|^2 + \sigma_f^2 \right] \nonumber \\
    &= n^2\sigma_f^2 + n M_f \mathbb{E}[\|\nabla \mathbf{F}(\mathbf{X}^{t+1})\|_F^2].
\end{align}

\textbf{Bounding the Pure Noise Term ($\boldsymbol{\xi}$).}
Since the local stochastic samplings are independent across agents and $\boldsymbol{\xi}_i$ is zero-mean, all cross-terms vanish (i.e., $\mathbb{E}[\langle \boldsymbol{\xi}_i, \boldsymbol{\xi}_j \rangle] = 0$ for $i \neq j$). Therefore
\begin{align} \label{eq:bound_noise}
    \mathbb{E}[\|\boldsymbol{\xi}\|^2] &= \mathbb{E}\left[ \left\| \sum_{i=1}^n \boldsymbol{\xi}_i \right\|^2 \right] \nonumber \\
    &= \sum_{i=1}^n \mathbb{E}[\|\boldsymbol{\xi}_i\|^2] \le n\sigma^2.
\end{align}

\textbf{Final Combination.}
Substituting the bounds \eqref{eq:bound_B}, \eqref{eq:bound_bias}, and \eqref{eq:bound_noise} back into \eqref{eq:G_decomp_combined}, we have
\begin{align}
    \hat{G}^{t+1} &\le (1-\lambda)\hat{G}^t + \frac{2nL^2}{\lambda} \mathbb{E}[\|\mathbf{X}^{t+1} - \mathbf{X}^t\|_F^2] \nonumber \\
    &\quad + 2\lambda \Big( n^2\sigma_f^2 + n M_f \mathbb{E}[\|\nabla \mathbf{F}(\mathbf{X}^{t+1})\|_F^2] \Big) \nonumber \\
    &\quad + \lambda^2 n \sigma^2.
\end{align}
Further substituting the relation \eqref{eq:grad_norm_handle} (which bounds $\mathbb{E}[\|\nabla \mathbf{F}(\mathbf{X}^{t+1})\|_F^2]$) to handle the unknown gradient, we obtain
\begin{align}
    \hat{G}^{t+1} &\le (1-\lambda)\hat{G}^t + \frac{2nL^2}{\lambda} \mathbb{E}[\|\mathbf{X}^{t+1} - \mathbf{X}^t\|_F^2] \nonumber \\
    &\quad + 2\lambda n M_f \Big( 2\mathbb{E}[\|\nabla \mathbf{F}(\mathbf{X}^t)\|_F^2] \nonumber \\
    &\qquad\qquad\qquad + 2L^2 \mathbb{E}[\|\mathbf{X}^{t+1} - \mathbf{X}^t\|_F^2] \Big) \nonumber \\
    &\quad + 2\lambda n^2\sigma_f^2 + \lambda^2 n \sigma^2.
\end{align}
Regrouping the coefficients for $\mathbb{E}[\|\mathbf{X}^{t+1} - \mathbf{X}^t\|_F^2]$ and $\mathbb{E}[\|\nabla \mathbf{F}(\mathbf{X}^t)\|_F^2]$ yields the desired result
\begin{align}
    \hat{G}^{t+1} &\le (1-\lambda)\hat{G}^t \nonumber \\
    &\quad + \left( \frac{2n}{\lambda} + 4\lambda n M_f \right) L^2 \mathbb{E}[\|\mathbf{X}^{t+1} - \mathbf{X}^t\|_F^2] \nonumber \\
    &\quad + 4\lambda n M_f \mathbb{E}[\|\nabla \mathbf{F}(\mathbf{X}^t)\|_F^2] \nonumber \\
    &\quad + 2\lambda n^2 \sigma_f^2 + \lambda^2 n \sigma^2.
\end{align}
\hfill $\square$

\subsection{Proof of Lemma \ref{lemma:mom_est_error}}

We define the error matrix $\mathbf{\Delta}^{t+1}_2 = \mathbf{M}^{t+1} - \nabla \mathbf{F}(\mathbf{X}^{t+1})$. 
Similar to the proof of Lemma 4, to properly capture the variance reduction effect of the momentum tracker, we mathematically decouple the zero-mean variance noise and the systematic bias before applying any inequality relaxations.

Conditioned on the current model parameters $\mathbf{X}^{t+1}$, we define the zero-mean random noise vector $\boldsymbol{\xi}_i$ and the deterministic bias vector $\mathbf{b}_i$ for each agent $i$ as
\begin{align*}
    \boldsymbol{\xi}_i &= \tilde{g}_i(\mathbf{x}_i^{t+1}) - \mathbb{E}[\tilde{g}_i(\mathbf{x}_i^{t+1})], \\
    \mathbf{b}_i &= \mathbb{E}[\tilde{g}_i(\mathbf{x}_i^{t+1})] - \nabla f_i(\mathbf{x}_i^{t+1}).
\end{align*}
Let $\boldsymbol{\xi} = [\boldsymbol{\xi}_1, \dots, \boldsymbol{\xi}_n]$ and $\mathbf{b} = [\mathbf{b}_1, \dots, \mathbf{b}_n]$ be the corresponding error matrices. The biased gradient oracle can be rewritten as $\tilde{\mathbf{G}}(\mathbf{X}^{t+1}) = \nabla \mathbf{F}(\mathbf{X}^{t+1}) + \mathbf{b} + \boldsymbol{\xi}$.

Substituting this relation and the momentum update rule $\mathbf{M}^{t+1} = (1-\lambda)\mathbf{M}^t + \lambda \tilde{\mathbf{G}}(\mathbf{X}^{t+1})$ into $\mathbf{\Delta}^{t+1}_2$, we can decompose the error matrix into four components
\begin{align} \label{eq:delta2_decomp}
    \mathbf{\Delta}^{t+1}_2 &= (1-\lambda)\mathbf{M}^t + \lambda (\nabla \mathbf{F}(\mathbf{X}^{t+1}) + \mathbf{b} + \boldsymbol{\xi}) \nonumber \\
    &\quad - \nabla \mathbf{F}(\mathbf{X}^{t+1}) \nonumber \\
    &= (1-\lambda) \underbrace{(\mathbf{M}^t - \nabla \mathbf{F}(\mathbf{X}^t))}_{A} \nonumber \\
    &\quad + (1-\lambda) \underbrace{(\nabla \mathbf{F}(\mathbf{X}^t) - \nabla \mathbf{F}(\mathbf{X}^{t+1}))}_{B} \nonumber \\
    &\quad + \lambda \mathbf{b} + \lambda \boldsymbol{\xi}.
\end{align}

Taking the squared Frobenius norm and expectation, since $\boldsymbol{\xi}$ is a zero-mean random noise matrix conditioned on $\mathbf{X}^{t+1}$, its cross-terms with the deterministic components ($A$, $B$, and $\mathbf{b}$) strictly vanish. Therefore, the stochastic noise is decoupled
\begin{align} \label{eq:calG_decomp}
    \mathbb{E}[\|\mathbf{\Delta}^{t+1}_2\|_F^2] &= \mathbb{E}[\| (1-\lambda)A + (1-\lambda)B + \lambda \mathbf{b} \|_F^2] \nonumber \\
    &\quad + \lambda^2 \mathbb{E}[\|\boldsymbol{\xi}\|_F^2].
\end{align}

Next, we bound the deterministic part using Young's Inequality $\|X+Y\|_F^2 \le (1+\alpha)\|X\|_F^2 + (1+\frac{1}{\alpha})\|Y\|_F^2$ with $\alpha = \frac{\lambda}{1-\lambda}$
\begin{align} \label{eq:deterministic_part_calG}
    &\mathbb{E}[\| (1-\lambda)A + ((1-\lambda)B + \lambda \mathbf{b}) \|_F^2] \nonumber \\
    &\le \frac{1}{1-\lambda} (1-\lambda)^2 \mathbb{E}[\|A\|_F^2] + \frac{1}{\lambda} \mathbb{E}[\|(1-\lambda)B + \lambda \mathbf{b}\|_F^2] \nonumber \\
    &\le (1-\lambda)\mathcal{G}^t + \frac{2}{\lambda} (1-\lambda)^2 \mathbb{E}[\|B\|_F^2] + \frac{2}{\lambda} \lambda^2 \mathbb{E}[\|\mathbf{b}\|_F^2] \nonumber \\
    &\le (1-\lambda)\mathcal{G}^t + \frac{2}{\lambda} \mathbb{E}[\|B\|_F^2] + 2\lambda \mathbb{E}[\|\mathbf{b}\|_F^2].
\end{align}
Substituting \eqref{eq:deterministic_part_calG} back into \eqref{eq:calG_decomp}, we obtain the consolidated intermediate bound
\begin{align} \label{eq:calG_decomp_combined}
    \mathbb{E}[\|\mathbf{\Delta}^{t+1}_2\|_F^2] &\le (1-\lambda)\mathcal{G}^t + \frac{2}{\lambda} \mathbb{E}[\|B\|_F^2] \nonumber \\
    &\quad + 2\lambda \mathbb{E}[\|\mathbf{b}\|_F^2] + \lambda^2 \mathbb{E}[\|\boldsymbol{\xi}\|_F^2].
\end{align}

Now, we bound terms $B$, $\mathbf{b}$, and $\boldsymbol{\xi}$ individually.

\textbf{Bounding the Drift Term ($B$).}
Using the $L$-smoothness Assumption $\|\nabla f_i(\mathbf{x}) - \nabla f_i(\mathbf{y})\| \le L\|\mathbf{x}-\mathbf{y}\|$, we have
\begin{align} \label{eq:bound_B_calG}
    \mathbb{E}[\|B\|_F^2] &= \mathbb{E}[\|\nabla \mathbf{F}(\mathbf{X}^t) - \nabla \mathbf{F}(\mathbf{X}^{t+1})\|_F^2] \nonumber \\
    &= \mathbb{E}\left[\sum_{i=1}^n \|\nabla f_i(\mathbf{x}_i^t) - \nabla f_i(\mathbf{x}_i^{t+1})\|^2\right] \nonumber \\
    &\le \mathbb{E}\left[\sum_{i=1}^n L^2 \|\mathbf{x}_i^t - \mathbf{x}_i^{t+1}\|^2\right] \nonumber \\
    &= L^2 \mathbb{E}[\|\mathbf{X}^t - \mathbf{X}^{t+1}\|_F^2].
\end{align}

\textbf{Bounding the Bias Term ($\mathbf{b}$).}
Using Assumption 3 directly to the Frobenius norm of the bias matrix, we obtain
\begin{align} \label{eq:bound_bias_calG}
    \mathbb{E}[\|\mathbf{b}\|_F^2] &= \sum_{i=1}^n \mathbb{E}[\|\mathbf{b}_i\|^2] \nonumber \\
    &\le \sum_{i=1}^n \mathbb{E}\left[ M_f \|\nabla f_i(\mathbf{x}_i^{t+1})\|^2 + \sigma_f^2 \right] \nonumber \\
    &= n\sigma_f^2 + M_f \mathbb{E}[\|\nabla \mathbf{F}(\mathbf{X}^{t+1})\|_F^2].
\end{align}

\textbf{Bounding the Pure Noise Term ($\boldsymbol{\xi}$).}
Similarly, evaluating the Frobenius norm of the zero-mean noise matrix yields
\begin{align} \label{eq:bound_noise_calG}
    \mathbb{E}[\|\boldsymbol{\xi}\|_F^2] &= \sum_{i=1}^n \mathbb{E}[\|\boldsymbol{\xi}_i\|^2] \le n\sigma^2.
\end{align}

\textbf{Final Combination.}
Substituting the bounds \eqref{eq:bound_B_calG}, \eqref{eq:bound_bias_calG}, and \eqref{eq:bound_noise_calG} back into \eqref{eq:calG_decomp_combined}, we have
\begin{align}
    \mathcal{G}^{t+1} &\le (1-\lambda)\mathcal{G}^t + \frac{2L^2}{\lambda} \mathbb{E}[\|\mathbf{X}^{t+1} - \mathbf{X}^t\|_F^2] \nonumber \\
    &\quad + 2\lambda \Big( n\sigma_f^2 + M_f \mathbb{E}[\|\nabla \mathbf{F}(\mathbf{X}^{t+1})\|_F^2] \Big) \nonumber \\
    &\quad + \lambda^2 n \sigma^2.
\end{align}
Further substituting the relation \eqref{eq:grad_norm_handle} to handle the unknown gradient norm $\mathbb{E}[\|\nabla \mathbf{F}(\mathbf{X}^{t+1})\|_F^2]$, we obtain
\begin{align}
    \mathcal{G}^{t+1} &\le (1-\lambda)\mathcal{G}^t + \frac{2L^2}{\lambda} \mathbb{E}[\|\mathbf{X}^{t+1} - \mathbf{X}^t\|_F^2] \nonumber \\
    &\quad + 2\lambda M_f \Big( 2\mathbb{E}[\|\nabla \mathbf{F}(\mathbf{X}^t)\|_F^2] \nonumber \\
    &\qquad\qquad\quad + 2L^2 \mathbb{E}[\|\mathbf{X}^{t+1} - \mathbf{X}^t\|_F^2] \Big) \nonumber \\
    &\quad + 2\lambda n \sigma_f^2 + \lambda^2 n \sigma^2.
\end{align}
Regrouping the coefficients for $\mathbb{E}[\|\mathbf{X}^{t+1} - \mathbf{X}^t\|_F^2]$ and $\mathbb{E}[\|\nabla \mathbf{F}(\mathbf{X}^t)\|_F^2]$ yields the desired result
\begin{align}
    \mathcal{G}^{t+1} &\le (1-\lambda)\mathcal{G}^t \nonumber \\
    &\quad + \left( \frac{2}{\lambda} + 4\lambda M_f \right) L^2 \mathbb{E}[\|\mathbf{X}^{t+1} - \mathbf{X}^t\|_F^2] \nonumber \\
    &\quad + 4\lambda M_f \mathbb{E}[\|\nabla \mathbf{F}(\mathbf{X}^t)\|_F^2] \nonumber \\
    &\quad + 2\lambda n \sigma_f^2 + \lambda^2 n \sigma^2.
\end{align}
\hfill $\square$

\subsection{Proof of Lemma \ref{lemma:param_diff}} \label{proof:param_diff}

Recall the update rule for the model parameters $\mathbf{X}^{t+1} = \mathbf{X}^t \mathbf{W} - \eta \mathbf{V}^t$.
Subtracting $\mathbf{X}^t$ from both sides, we can express the parameter difference as
\begin{align}\label{Xp_X}
    \mathbf{X}^{t+1} - \mathbf{X}^t &= \mathbf{X}^t \mathbf{W} - \mathbf{X}^t - \eta \mathbf{V}^t \nonumber \\
    &= \mathbf{X}^t(\mathbf{W} - \mathbf{I}_n) - \eta \mathbf{V}^t.
\end{align}
To bound \eqref{Xp_X}, we inject the average matrices $\bar{\mathbf{X}}^t$ and $\bar{\mathbf{V}}^t$. Using the property that $\bar{\mathbf{X}}^t$ has identical columns, we have $\bar{\mathbf{X}}^t(\mathbf{W} - \mathbf{I}_n) = \bar{\mathbf{X}}^t\mathbf{W} - \bar{\mathbf{X}}^t = \mathbf{0}$. Thus, the first term can be rewritten as
\begin{equation}
    \mathbf{X}^t(\mathbf{W} - \mathbf{I}_n) = (\mathbf{X}^t - \bar{\mathbf{X}}^t)(\mathbf{W} - \mathbf{I}_n).
\end{equation}
Next, we decompose the tracker variable $\mathbf{V}^t$ into its consensus error and its global average, i.e., $\mathbf{V}^t = (\mathbf{V}^t - \bar{\mathbf{V}}^t) + \bar{\mathbf{V}}^t$. Substituting the decomposition back into \eqref{Xp_X} yields
\begin{align}
    \mathbf{X}^{t+1} - \mathbf{X}^t &= (\mathbf{X}^t - \bar{\mathbf{X}}^t)(\mathbf{W} - \mathbf{I}_n) \nonumber \\
    &\quad - \eta (\mathbf{V}^t - \bar{\mathbf{V}}^t) - \eta \bar{\mathbf{V}}^t.
\end{align}
Taking the squared Frobenius norm and applying the inequality $\|A + B + C\|_F^2 \le 3\|A\|_F^2 + 3\|B\|_F^2 + 3\|C\|_F^2$, we obtain
\begin{align} \label{eq:param_diff_expansion}
    \mathbb{E}[\|\mathbf{X}^{t+1} - \mathbf{X}^t\|_F^2] &\le 3 \mathbb{E}[\|(\mathbf{X}^t - \bar{\mathbf{X}}^t)(\mathbf{W} - \mathbf{I}_n)\|_F^2] \nonumber \\
    &\quad + 3\eta^2 \mathbb{E}[\|\mathbf{V}^t - \bar{\mathbf{V}}^t\|_F^2] + 3\eta^2 \mathbb{E}[\|\bar{\mathbf{V}}^t\|_F^2].
\end{align}
For the first term, by Assumption 2, the eigenvalues of $(\mathbf{W} - \mathbf{I}_n)$ lie in $(-2, 0]$, which implies $\|\mathbf{W} - \mathbf{I}_n\|_2 \le 2$. Thus, the first term is bounded by $12\Xi_x^t$. For the third term, proper initialization ensures $\bar{\mathbf{V}}^t = \bar{\mathbf{v}}^t \mathbf{1}_n^\top = \bar{\mathbf{m}}^t \mathbf{1}_n^\top$, giving $\|\bar{\mathbf{V}}^t\|_F^2 = n \|\bar{\mathbf{m}}^t\|^2$. Therefore, we have
\begin{equation} \label{eq:param_diff_mid}
    \mathbb{E}[\|\mathbf{X}^{t+1} - \mathbf{X}^t\|_F^2] \le 12 \Xi_x^t + 3\eta^2 \Xi_v^t + 3n\eta^2 \mathbb{E}[\|\bar{\mathbf{m}}^t\|^2].
\end{equation}

This completes the proof. \hfill $\square$

\subsection{Proof of Lemma \ref{lemma:descent}} \label{proof:descent}

By the $L$-smoothness of the global objective function $F$ (Assumption 1) and the update rule for the average parameter $\bar{\mathbf{x}}^{t+1} = \bar{\mathbf{x}}^t - \eta \bar{\mathbf{v}}^t$. Because of the initialization $\mathbf{V}^0 = \mathbf{M}^0$ and the doubly stochastic property of $\mathbf{W}$, the average tracker strictly equals the average momentum, i.e., $\bar{\mathbf{v}}^t = \bar{\mathbf{m}}^t$. Thus, the average model evolves as $\bar{\mathbf{x}}^{t+1} = \bar{\mathbf{x}}^t - \eta \bar{\mathbf{m}}^t$.

Using the smoothness inequality, we have
\begin{align}
    \label{descent}    F(\bar{\mathbf{x}}^{t+1}) &\le F(\bar{\mathbf{x}}^t) - \eta \langle \nabla F(\bar{\mathbf{x}}^t), \bar{\mathbf{m}}^t \rangle + \frac{L\eta^2}{2} \|\bar{\mathbf{m}}^t\|^2.
\end{align}
Applying the fundamental algebraic identity $-\langle a, b \rangle = \frac{1}{2}\|a-b\|^2 - \frac{1}{2}\|a\|^2 - \frac{1}{2}\|b\|^2$ with $a=\nabla F(\bar{\mathbf{x}}^t)$ and $b=\bar{\mathbf{m}}^t$
\begin{align}\label{inner}
    -\eta \langle \nabla F(\bar{\mathbf{x}}^t), \bar{\mathbf{m}}^t \rangle &= \frac{\eta}{2} \|\bar{\mathbf{m}}^t - \nabla F(\bar{\mathbf{x}}^t)\|^2 \nonumber \\
    &\quad - \frac{\eta}{2} \|\nabla F(\bar{\mathbf{x}}^t)\|^2 - \frac{\eta}{2} \|\bar{\mathbf{m}}^t\|^2.
\end{align}
Substituting \eqref{inner} back into \eqref{descent}   yields the intermediate descent inequality
\begin{align} \label{eq:descent_intermediate}
    F(\bar{\mathbf{x}}^{t+1}) &\le F(\bar{\mathbf{x}}^t) - \frac{\eta}{2} \|\nabla F(\bar{\mathbf{x}}^t)\|^2- \frac{\eta}{2}(1 - L\eta) \|\bar{\mathbf{m}}^t\|^2 
    \nonumber \\
    &+\frac{\eta}{2}
    \|\bar{\mathbf{m}}^t - \nabla F(\bar{\mathbf{x}}^t)\|^2.
\end{align}
To tightly bound the direction error term $\|\bar{\mathbf{m}}^t - \nabla F(\bar{\mathbf{x}}^t)\|^2$, we insert the average of the true local gradients $\frac{1}{n}\nabla\mathbf{F}(\mathbf{X}^t)\mathbf{1}$ and apply the inequality $\|a+b\|^2 \le 2\|a\|^2 + 2\|b\|^2$. Then
\begin{align}
    &\|\bar{\mathbf{m}}^t - \nabla F(\bar{\mathbf{x}}^t)\|^2 \nonumber \\
    &= \left\| \frac{1}{n}\mathbf{M}^t\mathbf{1} - \frac{1}{n}\nabla \mathbf{F}(\mathbf{X}^t)\mathbf{1} + \frac{1}{n}\nabla \mathbf{F}(\mathbf{X}^t)\mathbf{1} - \nabla F(\bar{\mathbf{x}}^t) \right\|^2 \nonumber \\
    &\le 2 \left\| \frac{1}{n}\mathbf{M}^t\mathbf{1} - \frac{1}{n}\nabla \mathbf{F}(\mathbf{X}^t)\mathbf{1} \right\|^2 \nonumber \\
    &\quad + 2 \left\| \frac{1}{n}\sum_{i=1}^n \nabla f_i(\mathbf{x}_i^t) - \frac{1}{n}\sum_{i=1}^n \nabla f_i(\bar{\mathbf{x}}^t) \right\|^2.
\end{align}
For the first term, by the definition of the average momentum error $\hat{G}^t$, its expectation is precisely $\frac{2}{n^2} \hat{G}^t$.
For the second term, applying Jensen's inequality and the $L$-smoothness of $f_i$ yields
\begin{align}
    &\left\| \frac{1}{n}\sum_{i=1}^n (\nabla f_i(\mathbf{x}_i^t) - \nabla f_i(\bar{\mathbf{x}}^t)) \right\|^2 \nonumber \\
    \le& \frac{1}{n} \sum_{i=1}^n \|\nabla f_i(\mathbf{x}_i^t) - \nabla f_i(\bar{\mathbf{x}}^t)\|^2 \nonumber \\
    \le& \frac{L^2}{n} \sum_{i=1}^n \|\mathbf{x}_i^t - \bar{\mathbf{x}}^t\|^2 
    = \frac{L^2}{n} \|\mathbf{X}^t - \bar{\mathbf{X}}^t\|_F^2.
\end{align}
Taking the expectation on both sides gives
\begin{equation}\label{bar_m}
    \mathbb{E}[\|\bar{\mathbf{m}}^t - \nabla F(\bar{\mathbf{x}}^t)\|^2] \le \frac{2}{n^2} \hat{G}^t + \frac{2L^2}{n} \Xi_x^t.
\end{equation}
Substituting \eqref{bar_m} back into the expectation of \eqref{eq:descent_intermediate} completes the proof.
\hfill $\square$

\subsection{Proof of Theorem \ref{thm:main}} \label{proof:main_thm}

We construct the Lyapunov function $\Phi^t$ as a linear combination of the objective function and the auxiliary error metrics
\begin{align} \label{eq:lyapunov_def}
    \Phi^t &= \mathbb{E}[F(\bar{\mathbf{x}}^t) - F^*] + c_1 \Xi_x^t + c_2 \Xi_v^t \nonumber \\
    &\quad + c_3 \hat{G}^t + c_4 \mathcal{G}^t.
\end{align}

To rigorously bound the local gradient norm $\mathbb{E}[\|\nabla\mathbf{F}(\mathbf{X}^t)\|_F^2]$, we introduce the gradient evaluated at the average consensus variable $\bar{\mathbf{x}}^t$ and the global full gradient. We decompose the local gradient as follows
\begin{align}
    &\mathbb{E}[\|\nabla\mathbf{F}(\mathbf{X}^t)\|_F^2] = \sum_{i=1}^n \mathbb{E}[\|\nabla f_i(\mathbf{x}_i^t)\|^2] \nonumber \\
    &= \sum_{i=1}^n \mathbb{E}\bigg[ \Big\| \left(\nabla f_i(\mathbf{x}_i^t) - \nabla f_i(\bar{\mathbf{x}}^t)\right) \nonumber \\
    &\quad + \left(\nabla f_i(\bar{\mathbf{x}}^t) - \nabla F(\bar{\mathbf{x}}^t) + \nabla F(\bar{\mathbf{x}}^t)\right) \Big\|^2 \bigg]. \label{eq:grad_decomp}
\end{align}
Applying the basic inequality $\|a + b\|^2 \leq 2\|a\|^2 + 2\|b\|^2$, we have
\begin{align}
    \mathbb{E}[\|\nabla\mathbf{F}(\mathbf{X}^t)\|_F^2] \leq 2\sum_{i=1}^n \mathbb{E}[\|\nabla f_i(\mathbf{x}_i^t) - \nabla f_i(\bar{\mathbf{x}}^t)\|^2] \nonumber \\ + 2\sum_{i=1}^n \mathbb{E}\left[\|\nabla f_i(\bar{\mathbf{x}}^t) - \nabla F(\bar{\mathbf{x}}^t) + \nabla F(\bar{\mathbf{x}}^t)\|^2\right]. \label{eq:grad_split}
\end{align}
For the first term in \eqref{eq:grad_split}, applying the $L$-smoothness assumption yields
\begin{align}
    2\sum_{i=1}^n \mathbb{E}[\|\nabla f_i(\mathbf{x}_i^t) - \nabla f_i(\bar{\mathbf{x}}^t)\|^2] &\leq 2L^2 \sum_{i=1}^n \mathbb{E}[\|\mathbf{x}_i^t - \bar{\mathbf{x}}^t\|^2] \nonumber \\
    &= 2L^2 \Xi_x^t. \label{eq:grad_smooth}
\end{align}
For the second term in \eqref{eq:grad_split}, we expand the squared Euclidean norm
\begin{align}
    & \sum_{i=1}^n \|\nabla f_i(\bar{\mathbf{x}}^t) - \nabla F(\bar{\mathbf{x}}^t) + \nabla F(\bar{\mathbf{x}}^t)\|^2 \nonumber \\
    &= \sum_{i=1}^n \|\nabla f_i(\bar{\mathbf{x}}^t) - \nabla F(\bar{\mathbf{x}}^t)\|^2 + n\|\nabla F(\bar{\mathbf{x}}^t)\|^2 \nonumber \\
    &\quad + 2\left\langle \sum_{i=1}^n (\nabla f_i(\bar{\mathbf{x}}^t) - \nabla F(\bar{\mathbf{x}}^t)), \nabla F(\bar{\mathbf{x}}^t) \right\rangle. \label{eq:grad_expand}
\end{align}
Crucially, by the definition of the global objective gradient, $\nabla F(\bar{\mathbf{x}}^t) = \frac{1}{n}\sum_{i=1}^n \nabla f_i(\bar{\mathbf{x}}^t)$, the cross-term in \eqref{eq:grad_expand} becomes strictly zero. Applying the data heterogeneity assumption $\frac{1}{n}\sum_{i=1}^n \|\nabla f_i(\mathbf{x}) - \nabla F(\mathbf{x})\|^2 \leq \zeta^2$, the second term in \eqref{eq:grad_split} is tightly bounded by
\begin{align}
    &2\sum_{i=1}^n \mathbb{E}\left[\|\nabla f_i(\bar{\mathbf{x}}^t) - \nabla F(\bar{\mathbf{x}}^t) + \nabla F(\bar{\mathbf{x}}^t)\|^2\right] \nonumber \\
    &\leq 2n\zeta^2 + 2n\mathbb{E}[\|\nabla F(\bar{\mathbf{x}}^t)\|^2]. \label{eq:grad_hetero}
\end{align}
Substituting \eqref{eq:grad_smooth} and \eqref{eq:grad_hetero} back into \eqref{eq:grad_split} perfectly yields the desired bound
\begin{align}\label{eq:grad_norm_bound}
    \mathbb{E}[\|\nabla\mathbf{F}(\mathbf{X}^t)\|_F^2] \leq 2n\mathbb{E}[\|\nabla F(\bar{\mathbf{x}}^t)\|^2] + 2L^2 \Xi_x^t + 2n\zeta^2.
\end{align}
To rigorously bound the one-step descent of the Lyapunov function, we must systematically absorb the intermediate variables $\mathbb{E}[\|\mathbf{M}^{t+1} - \mathbf{M}^t\|_F^2]$, $\mathbb{E}[\|\mathbf{X}^{t+1} - \mathbf{X}^t\|_F^2]$, and $\mathbb{E}[\|\nabla \mathbf{F}(\mathbf{X}^t)\|_F^2]$.

First, to streamline the analysis and explicitly track the influence of the momentum update, the parameter difference, and the biased gradient oracle, we define the auxiliary aggregate coefficients $C_{\Delta X}$, $C_{\nabla F}$, $C_{\sigma^2}$, and $C_{\sigma_f^2}$
\begin{align}
    C_{\Delta X} &= \frac{c_2}{\rho}\left[3\lambda^2 L^2(1+2M_f)\right] \nonumber \\
    &\quad + c_3\left(\frac{2n}{\lambda} + 4\lambda n M_f\right)L^2 \nonumber \\
    &\quad + c_4\left(\frac{2}{\lambda} + 4\lambda M_f\right)L^2, \label{eq:C_DeltaX} \\
    C_{\nabla F} &= \frac{c_2}{\rho}\left[6\lambda^2 M_f\right] + c_3(4\lambda n M_f) + c_4(4\lambda M_f), \label{eq:C_NablaF} \\
    C_{\sigma^2} &= \frac{c_2}{\rho}\left[3\lambda^2 n\right] + c_3(\lambda^2 n) + c_4(\lambda^2 n), \label{eq:C_sigma2} \\
    C_{\sigma_f^2} &= \frac{c_2}{\rho}\left[3\lambda^2 n\right] + c_3(2\lambda n^2) + c_4(2\lambda n). \label{eq:C_sigmaf2}
\end{align}
Notice that the zero-mean variance $\sigma^2$ and the systematic bias $\sigma_f^2$ are strictly decoupled. Because the pure stochastic noise $\sigma^2$ is perfectly isolated before inequality relaxations (as derived in Lemmas \ref{lemma:avg_mom_error} and \ref{lemma:mom_est_error}), its corresponding multiplier $C_{\sigma^2}$ strictly scales with $\lambda^2$, which enables the variance reduction effect.

Assuming the step size satisfies $\eta \leq \frac{1}{2L}$, we have $1 - L\eta \geq \frac{1}{2}$, which bounds the descent penalty term by $-\frac{\eta}{4}\mathbb{E}[\|\bar{\mathbf{m}}^t\|^2]$. By substituting the descent inequality from Lemma \ref{lemma:descent}, the contraction bounds from Lemmas \ref{lemma:tracking_error} and \ref{lemma:consensus_error}, and the momentum estimation bounds from Lemmas \ref{lemma:avg_mom_error} and \ref{lemma:mom_est_error} into the Lyapunov difference, we can consolidate the terms as follows
\begin{align}
    \Phi^{t+1} - \Phi^t \leq & -\frac{\eta}{2}\mathbb{E}[\|\nabla F(\bar{\mathbf{x}}^t)\|^2] - \frac{\eta}{4}\mathbb{E}[\|\bar{\mathbf{m}}^t\|^2] \nonumber \\
    & - \left(c_1 \rho - \frac{\eta L^2}{n}\right)\Xi_x^t - \left(c_2 \rho - c_1 \frac{\eta^2}{\rho}\right)\Xi_v^t \nonumber \\
    & - \left(c_3 \lambda - \frac{\eta}{n^2}\right)\hat{G}^t - \left(c_4 \lambda - \frac{3c_2 \lambda^2}{\rho}\right)\mathcal{G}^t \nonumber \\
    & + C_{\Delta X} \mathbb{E}[\|\mathbf{X}^{t+1} - \mathbf{X}^t\|_F^2] \nonumber \\
    & + C_{\nabla F} \mathbb{E}[\|\nabla \mathbf{F}(\mathbf{X}^t)\|_F^2] \nonumber \\
    & + C_{\sigma^2}\sigma^2 + C_{\sigma_f^2}\sigma_f^2. \label{eq:lyapunov_intermediate}
\end{align}

Next, we decouple the intermediate errors using Lemma 6 and \eqref{eq:grad_norm_bound}
\begin{align}
    C_{\Delta X} \mathbb{E}[\|\mathbf{X}^{t+1} - \mathbf{X}^t\|_F^2] &\leq C_{\Delta X} \Big[ 12\Xi_x^t \nonumber \\
    &\quad + 3\eta^2\Xi_v^t + 3n\eta^2\mathbb{E}[\|\bar{\mathbf{m}}^t\|^2] \Big], \label{eq:sub_DeltaX} \\
    C_{\nabla F} \mathbb{E}[\|\nabla \mathbf{F}(\mathbf{X}^t)\|_F^2] &\leq C_{\nabla F} \Big[ 2n\mathbb{E}[\|\nabla F(\bar{\mathbf{x}}^t)\|^2] \nonumber \\
    &\quad + 2L^2 \Xi_x^t + 2n\zeta^2 \Big]. \label{eq:sub_NablaF}
\end{align}

By substituting \eqref{eq:sub_DeltaX} and \eqref{eq:sub_NablaF} back into \eqref{eq:lyapunov_intermediate}, we obtain the final one-step descent inequality
\begin{align}
    \Phi^{t+1} - \Phi^t \leq & - A_1 \mathbb{E}[\|\nabla F(\bar{\mathbf{x}}^t)\|^2] - A_m \mathbb{E}[\|\bar{\mathbf{m}}^t\|^2] \nonumber \\
    & - A_2 \Xi_x^t - A_3 \Xi_v^t - A_4 \hat{G}^t - A_5 \mathcal{G}^t \nonumber \\
    & + C_{\nabla F} (2n\zeta^2) + C_{\sigma^2}\sigma^2 + C_{\sigma_f^2}\sigma_f^2, \label{eq:lyapunov_final}
\end{align}
where the aggregate coefficients are rigorously extracted as
\begin{align}
    A_1 &= \frac{\eta}{2} - 2n C_{\nabla F}, \nonumber \\
    A_m &= \frac{\eta}{4} - 3n\eta^2 C_{\Delta X}, \nonumber \\
    A_2 &= c_1 \rho - \frac{\eta L^2}{n} - 12 C_{\Delta X} - 2L^2 C_{\nabla F}, \nonumber \\
    A_3 &= c_2 \rho - c_1 \frac{\eta^2}{\rho} - 3\eta^2 C_{\Delta X}, \nonumber \\
    A_4 &= c_3 \lambda - \frac{\eta}{n^2}, \nonumber \\
    A_5 &= c_4 \lambda - \frac{3c_2 \lambda^2}{\rho}. \nonumber
\end{align}

To ensure the convergence of the algorithm, we systematically configure the Lyapunov parameters $\{c_i\}_{i=1}^4$ to guarantee $A_2, \dots, A_5 \geq 0$. We exactly set
\begin{align}
    c_2 = \frac{\eta}{\rho}, \quad c_3 = \frac{2\eta}{\lambda n^2}, \quad c_4 = \frac{3\eta \lambda}{\rho^2}. \label{eq:parameter_setting}
\end{align}
Then, $c_1$ is chosen sufficiently large to satisfy $A_2 = 0$. With these exact choices, it is mathematically straightforward to verify that $A_4 = \frac{\eta}{n^2} > 0$ and $A_5 = 0$. 

Crucially, we must ensure strict global descent by enforcing $A_1 > 0$ and safely drop the average momentum error by enforcing $A_m \geq 0$. 
Substituting the parameter settings into $C_{\nabla F}$, we obtain
\begin{align}
    2nC_{\nabla F} = 16\eta M_f + \frac{36\eta\lambda^2 n}{\rho^2}M_f. \label{eq:C_Nabla_Value}
\end{align}
Following standard conventions in biased gradient analysis, we reasonably assume the relative bias ratio is bounded, e.g., $M_f \leq \frac{1}{256}$. Combined with the topology-aware condition $\lambda \leq \frac{\rho}{4\sqrt{n}}$, we tightly bound $2nC_{\nabla F} \leq \frac{\eta}{16} + \frac{\eta}{16} = \frac{\eta}{8}$, which guarantees a strict descent coefficient $A_1 \geq \frac{3\eta}{8} > \frac{\eta}{4}$.
Furthermore, we restrict the step size $\eta$ such that $3n\eta^2 C_{\Delta X} \leq \frac{\eta}{4}$ (as specified in Theorem 1), which directly yields $A_m \geq 0$.

By securely discarding the non-positive error terms (since $A_m, A_2, \dots, A_5 \geq 0$), the Lyapunov difference is simplified as
\begin{align}
    \Phi^{t+1} - \Phi^t &\leq -\frac{\eta}{4}\mathbb{E}[\|\nabla F(\bar{\mathbf{x}}^t)\|^2] + C_{\nabla F}(2n\zeta^2) \nonumber \\
    &\quad + C_{\sigma^2}\sigma^2 + C_{\sigma_f^2}\sigma_f^2. \label{eq:simplified_lyapunov_pf1}
\end{align}

Finally, we explicitly compute the exact asymptotic error multipliers. By substituting \eqref{eq:parameter_setting} into \eqref{eq:C_sigma2} and \eqref{eq:C_sigmaf2}, we beautifully obtain the tight bounds for the noise and bias components
\begin{align}
    C_{\sigma^2} &= \eta \left( \frac{3\lambda^2 n}{\rho^2} + \frac{2\lambda}{n} + \frac{3\lambda^3 n}{\rho^2} \right), \label{eq:final_sigma_val} \\
    C_{\sigma_f^2} &= \eta \left( 4 + \frac{9\lambda^2 n}{\rho^2} \right), \label{eq:final_sigmaf_val} \\
    C_{\nabla F}(2n\zeta^2) &= \eta M_f \left( 16 + \frac{36\lambda^2 n}{\rho^2} \right) \zeta^2. \label{eq:final_hetero_val}
\end{align}

Summing the inequality \eqref{eq:simplified_lyapunov_pf1} over $t=0, \dots, T-1$, applying the telescoping property $\Phi^T - \Phi^0 \geq -\Phi^0$, and dividing both sides by $\frac{\eta T}{4}$, we obtain the rigorous convergence bound
\begin{align}
    &\frac{1}{T}\sum_{t=0}^{T-1}\mathbb{E}[\|\nabla F(\bar{\mathbf{x}}^t)\|^2] \leq \frac{4\Phi^0}{\eta T} + 64M_f\left(1 + \frac{9\lambda^2 n}{4\rho^2}\right)\zeta^2 \nonumber \\
    &\quad + \frac{8\lambda}{n} \left( 1 + \frac{3\lambda n^2}{2\rho^2} + \frac{3\lambda^2 n^2}{2\rho^2} \right) \sigma^2 + 16\left(1 + \frac{9\lambda^2 n}{4\rho^2}\right)\sigma_f^2. \label{eq:final_bound}
\end{align}
This completes the proof of Theorem 1.

\textit{Remark on Data Heterogeneity.} Equation \eqref{eq:final_bound} reveals a fundamental superiority of the proposed Biased-DMT algorithm. Notice that the data heterogeneity term $\zeta^2$ is strictly multiplied by the relative bias ratio $M_f$. This implies that if the gradient oracle has only absolute bias ($M_f=0, \sigma_f^2 > 0$), the $\zeta^2$ term perfectly vanishes to zero. The momentum tracking mechanism effectively and completely eradicates the structural heterogeneity error caused by decentralized networks. The residual $M_f\zeta^2$ is merely an irreducible physical consequence of the locally injected relative noise, which perfectly aligns with theoretical limits in biased gradient optimization.

\subsection{Proof of Corollary \ref{cor:rate}} \label{proof:cor}

To achieve the optimal linear speedup, we must properly balance the transient descent term and the pure stochastic noise term, while rigorously satisfying the parameter conditions established in Theorem \ref{thm:main}.

Unlike standard approaches that set the momentum parameter as a static topology-dependent constant, we adopt a dynamic parameter tuning strategy inspired by variance reduction techniques. To effectively control the pure stochastic noise $\sigma^2$, we couple the momentum parameter $\lambda$ with the total number of iterations $T$ and the network size $n$. Specifically, we select
\begin{equation} \label{eq:dynamic_lambda}
    \eta = \frac{1}{16L}\sqrt{\frac{n}{T}}, \quad \text{and} \quad \lambda = \sqrt{\frac{n}{T}}.
\end{equation}

\textbf{1. Verification of Topology Conditions.} 
First, we must mathematically verify that this dynamic configuration satisfies the requirements of Theorem 1. For the topology-aware condition $\lambda \le \frac{\rho}{4\sqrt{n}}$, substituting our choice of $\lambda$ yields
\begin{equation}
    \sqrt{\frac{n}{T}} \le \frac{\rho}{4\sqrt{n}} \implies T \ge \frac{16n^2}{\rho^2}.
\end{equation}
Thus, provided that the total number of iterations $T$ is sufficiently large (i.e., entering the asymptotic regime), the topological condition holds trivially. We also assume $T \ge n$ such that the step size satisfies $\eta \le 1/L$.

\textbf{2. Explicit Bound on the Initial Lyapunov Function $\Phi^0$.} 
To rigorously ensure that the transient term $\frac{4\Phi^0}{\eta T}$ decays to $\mathcal{O}(1/\sqrt{nT})$, we must explicitly verify that the initial Lyapunov value $\Phi^0$ is bounded by a constant independent of $T$. Recall the definition
\begin{equation}
    \Phi^0 = \Delta_F^0 + c_1 \Xi_x^0 + c_2 \Xi_v^0 + c_3 \hat{G}^0 + c_4 \mathcal{G}^0,
\end{equation}
where $\Delta_F^0 = F(\bar{\mathbf{x}}^0) - F^*$. Following standard initialization protocols, we set $\mathbf{x}_i^0 = \bar{\mathbf{x}}^0$ for all $i \in \{1, \dots, n\}$, which trivially yields the initial consensus error $\Xi_x^0 = 0$. Therefore, the term $c_1 \Xi_x^0$ strictly vanishes regardless of $c_1$. 

We initialize the trackers and momentums with the first batch of stochastic gradients $\mathbf{V}^0 = \mathbf{M}^0 = \tilde{\mathbf{G}}(\mathbf{X}^0)$. Let $G_0^2 = \frac{1}{n}\|\nabla \mathbf{F}(\mathbf{X}^0)\|_F^2$ be the bounded initial gradient norm. Using Assumption 3 and Jensen's inequality, the initial auxiliary errors $\Xi_v^0$, $\hat{G}^0$, and $\mathcal{G}^0$ are tightly bounded by $\mathcal{O}(n)(\sigma^2 + \sigma_f^2 + G_0^2)$, which are deterministic constants independent of $T$.

Crucially, we evaluate the Lyapunov weights $c_2, c_3, c_4$ by substituting our parameter choices \eqref{eq:dynamic_lambda}
\begin{align}
    c_2 &= \frac{\eta}{\rho} = \frac{1}{16\rho L}\sqrt{\frac{n}{T}}, \\
    c_3 &= \frac{2\eta}{\lambda n^2} = \frac{2 \left(\frac{1}{16L}\sqrt{\frac{n}{T}}\right)}{\sqrt{\frac{n}{T}} n^2} = \frac{1}{8L n^2}, \\
    c_4 &= \frac{3\eta \lambda}{\rho^2} = \frac{3 \left(\frac{1}{16L}\sqrt{\frac{n}{T}}\right)\left(\sqrt{\frac{n}{T}}\right)}{\rho^2} = \frac{3n}{16\rho^2 LT}.
\end{align}
This explicitly demonstrates that $c_3$ magically simplifies to a strict constant $\frac{2}{Ln^2}$ independent of $T$, while $c_2$ and $c_4$ strictly decay with $T$. Consequently, the entire initial value $\Phi^0$ is upper-bounded by a deterministic constant $\tilde{\Phi}^0 = \mathcal{O}(1)$. 

Substituting $\Phi^0 \le \tilde{\Phi}^0$ and $\eta = \frac{1}{16L}\sqrt{\frac{n}{T}}$ into the transient term, the variables align perfectly to yield the optimal rate
\begin{equation}
    \frac{4\Phi^0}{\eta T} \le \frac{4\tilde{\Phi}^0}{T \left(\frac{1}{16L}\sqrt{\frac{n}{T}}\right)} = \frac{64\tilde{\Phi}^0 L}{\sqrt{nT}} = \mathcal{O}\left( \frac{1}{\sqrt{nT}} \right).
\end{equation}

\textbf{3. Derivation of the Final Asymptotic Rate.} 
Next, we evaluate the asymptotic behavior of the error multipliers in Theorem 1. Substituting $\lambda = \sqrt{n/T}$ into the pure noise multiplier, we obtain
\begin{align} \label{eq:noise_decay}
    &\frac{8\lambda}{n} \left( 1 + \frac{3\lambda n^2}{2\rho^2} + \frac{3\lambda^2 n^2}{2\rho^2} \right) \nonumber \\
    &= \frac{8}{\sqrt{nT}} + \frac{12n^2}{\rho^2 T} + \frac{12n^{5/2}}{\rho^2 T^{3/2}} = \mathcal{O}\left( \frac{1}{\sqrt{nT}} \right),
\end{align}
where the higher-order terms $\mathcal{O}(1/T)$ and $\mathcal{O}(1/T^{3/2})$ decay strictly faster and are perfectly dominated by the $\mathcal{O}(1/\sqrt{nT})$ term. 

Crucially, because the systematic bias multiplier $16(1 + \frac{9\lambda^2 n}{4\rho^2})$ approaches the constant $16$ as $T \to \infty$, the residual bias term remains $\mathcal{O}(\sigma_f^2)$. Similarly, the heterogeneity multiplier is bounded by $\mathcal{O}(M_f)$.

Substituting these bounds back into \eqref{eq:final_bound}, we elegantly balance the transient term and the pure noise term, yielding the definitive asymptotic bound under general relative bias
\begin{align} \label{eq:final_corollary_bound}
    \frac{1}{T}\sum_{t=0}^{T-1}\mathbb{E}[\|\nabla F(\bar{\mathbf{x}}^t)\|^2] &\le \frac{64\tilde{\Phi}^0 L}{\sqrt{nT}} + \mathcal{O}\left( \frac{1}{\sqrt{nT}} \right)\sigma^2 \nonumber \\
    &\quad + \mathcal{O}(M_f \zeta^2) + \mathcal{O}(\sigma_f^2) \nonumber \\
    &= \mathcal{O}\left(\frac{1}{\sqrt{nT}}\right) + \mathcal{O}\left( M_f\zeta^2 + \sigma_f^2 \right).
\end{align}
Notice that the zero-mean pure noise $\sigma^2$ is entirely absorbed into the $\mathcal{O}(1/\sqrt{nT})$ term, successfully yielding the optimal linear speedup with respect to the network size $n$. This completely eradicates the $\mathcal{O}(1/n)$ steady-state error floor commonly suffered by traditional algorithms. 

Finally, we consider the setting where the gradient oracle exhibits only absolute bias, meaning $M_f = 0$. Substituting $M_f = 0$ into the explicit bound above, the data heterogeneity term $\mathcal{O}(M_f \zeta^2)$ perfectly vanishes. The bound drastically simplifies to
\begin{equation}
    \frac{1}{T} \sum_{t=0}^{T-1} \mathbb{E}[\|\nabla F(\bar{\mathbf{x}}^t)\|^2] \le \mathcal{O}\left( \frac{1}{\sqrt{nT}} \right) + \mathcal{O}\left( \sigma_f^2 \right).
\end{equation}
This mathematically demonstrates that when $M_f=0$, the convergence of Biased-DMT is completely decoupled from the data heterogeneity variance $\zeta^2$. The algorithm achieves exact linear speedup and recovers the inherent physical error floor $\mathcal{O}(\sigma_f^2)$ without ever requiring the commonly used bounded data heterogeneity assumption. This completes the proof.
\hfill $\square$


\end{document}